\documentclass[11pt]{article}
\usepackage[margin=0.87in]{geometry}
\usepackage{float}
\usepackage{hyperref}
\usepackage{titling}
\usepackage{blindtext}
\usepackage{amsmath}
\usepackage{amssymb}
\usepackage{mathtools}
\usepackage[numbers]{natbib}
\usepackage{tikz}
\usepackage{graphicx}
\usepackage{siunitx}
\usepackage{gensymb}
\usepackage{url}
\usepackage{mathrsfs}
\usepackage{amsthm}
\usepackage{cleveref}
\usepackage{tkz-euclide}
\usepackage{titlesec}
\usepackage{url, verbatim}
\usepackage{enumerate}
\usepackage{pgfplots}
\usetikzlibrary{calc,patterns,quotes,shapes,arrows,through,intersections,decorations.pathreplacing}
\allowdisplaybreaks

\newtheorem{thm}{Theorem}[section]

\newtheorem{lem}[thm]{Lemma}
\newtheorem{prop}[thm]{Proposition}

\numberwithin{equation}{section}

\newcommand{\mathr}{\mathbb{R}}
\newcommand{\mathn}{\mathbb{N}}

\newcommand{\mathz}{\mathbb{Z}}
\newcommand{\muball}[2]{\mu(B(#1,e^{-#2}))}
\newcommand{\ubox}{\overline{\text{dim}}_\text{B}}
\newcommand{\lbox}{\underline{\text{dim}}_\text{B}}
\newcommand{\boxd}{\text{dim}_\text{B}}
\newcommand{\aso}{{\text{dim}}_\text{A}}

\newcommand{\low}{{\text{dim}}_\text{L}}
\newcommand{\haus}{{\text{dim}}_\text{H}}

\newcommand{\asospec}{{\text{dim}}^{\theta}_\text{A}}
\newcommand{\lowspec}{{\text{dim}}^{\theta}_\text{L}}

\newcommand{\lset}{L(\Gamma)}
\newcommand{\kmin}{k_{\min}}
\newcommand{\kmax}{k_{\max}}

\newcommand{\hdist}[2]{d_\mathbb{H}(#1, #2)}

\newcommand{\size}[1]{\vert #1 \vert}

\renewcommand{\epsilon}{\varepsilon}

\renewcommand{\geq}{\geqslant}
\renewcommand{\leq}{\leqslant}

\definecolor{lightgray}{rgb}{0.83, 0.83, 0.83}

\title{The Assouad spectrum of Kleinian limit sets \\  and Patterson-Sullivan measure}
\author{Jonathan M. Fraser and Liam Stuart \\ \\
The University of St Andrews, Scotland\\
E-mails: jmf32@st-andrews.ac.uk and  ls220@st-andrews.ac.uk}

\newlength{\bibitemsep}
\newlength{\bibparskip}
\let\oldthebibliography\thebibliography
\renewcommand\thebibliography[1]{%
  \oldthebibliography{#1}%
  \setlength{\parskip}{\bibitemsep}%
  \setlength{\itemsep}{\bibparskip}%
}

\date{}

\begin{document}
\pagenumbering{arabic}
\maketitle
\begin{abstract}
The Assouad dimension of the limit set of a geometrically finite Kleinian group with parabolics  may exceed the Hausdorff and box dimensions.  The Assouad \emph{spectrum} is a continuously parametrised family of dimensions which `interpolates' between the box and Assouad dimensions of a fractal set.  It is designed to reveal more subtle geometric information than the box and Assouad dimensions considered in isolation.  We conduct a detailed analysis of the  Assouad   spectrum  of limit sets of geometrically finite Kleinian groups and the associated Patterson-Sullivan measure.   Our analysis  reveals several novel features, such as  interplay between horoballs of different rank  not seen by the box or Assouad dimensions.  
 

 \end{abstract}

\section{Introduction}\label{Introduction}

Non-elementary Kleinian groups   generate important examples of dynamically invariant fractal sets living on the boundary of hyperbolic space.   Seminal work of Patterson, Sullivan and others established that the Hausdorff, packing and box dimensions of the limit set coincide   in the geometrically finite case and are given by the associated Poincar\'e exponent, denoted by $\delta$. This is perhaps remarkable because   dimension is a fine measure of how  the limit set takes up space on small scales whereas the Poincar\'e exponent is a  coarse measure of the rate of accumulation of orbits to the boundary.  The Assouad dimension is another notion of fractal dimension, which first came to prominence due to the central role it played in, for example, embedding theory and conformal geometry, see \cite{MT, R}.  However, it is rapidly gaining prominence in the literature on fractal geometry and the dimension theory of dynamical systems, see \cite{Fr2}.  Fraser \cite{Fr1} proved that the Assouad dimension of the limit set of a geometrically finite Kleinian group is not necessarily given by the Poincar\'e exponent, but is instead given by $\max\{\delta, \kmax\}$ where $\kmax$ is the maximal rank of a parabolic fixed point.  This result is the starting point of our work, which attempts to understand the gap in-between the box dimension and the Assouad dimension by considering the, more recently  introduced, Assouad \emph{spectrum}.  This spectrum is a continuum of dimensions which  interpolate between the box and Assouad dimensions in a meaningful sense, providing more nuanced information about the scaling structure of the fractal object at hand.   This is part of a more general programme of `dimension interpolation', which is proving to be a useful concept. For example, the Assouad spectrum has found relevance in surprising contexts such as  the study of certain spherical maximal functions \cite{AHRS, RS}.  We compute the Assouad spectrum and (its natural dual) the lower spectrum for limit sets of non-elementary geometrically finite Kleinian groups and the associated Patterson-Sullivan measure.  These results shed some new light on the `Sullivan dictionary' in the context of dimension theory, see \cite{stuartsurvey}.

Our proofs use a variety of techniques.  We use the \emph{global measure formula} which allows  us to utilise  the Patterson-Sullivan measure to estimate the cardinality of efficient covers.  We take some inspiration from the paper \cite{Fr1} which dealt with the Assouad and lower dimensions of Kleinian limit sets.  However, the Assouad and lower spectra require much finer control and therefore many of the techniques from \cite{Fr1} need refined and some need replaced. Since we consider several dual notions of dimension, some of the arguments are analogous and we do our best to suppress repetition.  We stress, however, that calculating the lower spectrum (for example) is not usually a  case of simply `reversing' the Assouad spectrum arguments and subtle differences often emerge. For example, Bowditch's  theorem describing the geometry near  parabolic fixed points is only needed to study the lower spectrum.

For notational convenience, we write $A \lesssim B$ if there exists a constant $C \geq 1$ such that $A \leq CB$, and $A \gtrsim B$ if $B \lesssim A$. We write $A \approx B$ if $A \lesssim B$ and $B \lesssim A$. The constant $C$ is allowed to depend on parameters fixed in the hypotheses of the theorems presented, but not on parameters introduced in the proofs.

\section{Definitions and Background}\label{Prelims}
\subsection{Dimensions of sets and measures}
\label{DimPrelims}
We recall the key notions from fractal geometry and dimension theory which we  use throughout the paper.  For a more in-depth treatment see the books \cite{BP, FK} for background on Hausdorff and box dimensions, and \cite{Fr2} for Assouad type dimensions. Kleinian limit sets will be subsets of the $d$-dimensional sphere $\mathbb{S}^d$ which we  view as a subset of $\mathbb{R}^{d+1}$. Therefore, it is convenient to recall dimension theory in Euclidean space only. 

Let $F \subseteq \mathbb{R}^d$ be non-empty. We  write $\aso F$, $\low F$ and $\haus F$ to denote the Assouad, lower,  and Hausdorff dimension of $F$, respectively.  We also write $\ubox F$ and $\lbox F$ for the upper and lower box dimensions of $F$ and $\boxd F$ for the box dimension when it exists.  We refer the reader to \cite{BP, FK, Fr2} for the precise definitions since we do not use them directly. It is useful to keep in mind that, for compact $F$,
\[
\low F \leq \haus F \leq \lbox F \leq \ubox F \leq \aso F.
\]
We write 
\[
\size{F} = \sup_{x,y \in F} |x-y| \in [0,\infty]
\]
 to denote the diameter of  $F$.  Given $r>0$, we write $N_r(F)$ to denote the smallest number of balls of radius $r$ required to cover $F$. We write $M_r(F)$ to denote the largest cardinality of a packing of $F$ by balls of radius $r$ centred in $F$.  In what follows, it is easy to see that  replacing  $N_r(F)$   by $M_r(F)$  yields  equivalent definitions and so we sometimes switch between minimal coverings and maximal packings in our arguments. This is standard in fractal geometry.

The Assouad and lower spectra, introduced in \cite{FYu}, interpolate between the box dimensions and the Assouad and lower dimensions in a meaningful way.  They provide a parametrised family of dimensions by fixing the relationship between the two scales  used to define Assouad and lower dimension, see below.   Studying the dependence on the parameter within this family  thus yields finer and more nuanced information about the local structure of the set.  For example, one may understand which scales `witness' the behaviour described by the Assouad and lower dimensions.  For $\theta \in (0,1)$, the \textit{Assouad spectrum} of $F$ is given by
\begin{align*}
\asospec F = \inf \Bigg\{ s \geq 0 \mid  \exists C>0 \ : \ \forall \ 0<r<1 \  :  \ \forall x \in F \  :  \ 
N_r(B(x,r^{\theta}) \cap F) \leq C \left(\frac{r^{\theta}}{r} \right)^{s} \Bigg\} 
\end{align*}
and the \textit{lower spectrum} of $F$ by
\begin{align*}
\lowspec F = \sup \Bigg\{ s \geq 0 \mid  \exists C>0 \ : \ \forall \ 0<r<1 \  :  \ \forall x \in F \  :  \  
N_r(B(x,r^{\theta}) \cap F) \geq C \left(\frac{r^{\theta}}{r} \right)^{s} \Bigg\}. 
\end{align*}
If one replaces $r^\theta$ in the above definitions with a `free scale'  $R \in (r,1)$, then one recovers the Assouad and lower dimensions, respectively.  It was shown in \cite{FYu} that for a bounded set $F \subseteq \mathbb{R}^d$,
\begin{align}
\ubox F &\leq \asospec F \leq \min\left\{\aso F, \ \frac{\ubox F}{1-\theta}\right\} \label{basicbound}\\
\low F &\leq \lowspec F \leq \lbox F. \nonumber
\end{align}
In particular, $\asospec F \to \ubox F$ as $\theta \to 0$.  Whilst the analogous statement does not hold for the lower spectrum in general, it was proved in \cite[Theorem 6.3.1]{Fr2} that $\lowspec F \to \lbox F$ as $\theta \to 0$ provided $F$ satisfies a strong form of  dynamical invariance.   Whilst the fractals we study are not quite covered by this result, we shall see that this `interpolation' holds nevertheless.    The limits  $\lim_{\theta \to 1} \asospec F$ and $\lim_{\theta \to 1} \lowspec F$ are known to exist in general, but are not necessarily equal to the  Assouad and lower dimensions, respectively, although we shall see this will hold for the sets considered here.

There is an analogous dimension theory of measures, and the interplay between the dimension theory of fractals and the measures they support is fundamental to fractal geometry, especially in the dimension theory of dynamical systems. Let $\mu$ be a locally finite Borel measure on $\mathbb{R}^d$. Similar to above, we write $\aso \mu$, $\low \mu$ and $\haus \mu$ for the Assouad, lower and (lower) Hausdorff dimensions of  $\mu$, respectively.  For $\theta \in (0,1)$, the \textit{Assouad spectrum} of $\mu$ with support $F$ is given by
\begin{align*}
\asospec \mu = \inf \Bigg\{ s \geq 0 \mid  \exists C>0 \ : \ \forall \ 0<r< \vert F \vert \  :  \ \forall x \in F \  :  \  \frac{\mu(B(x,r^\theta))}{\mu(B(x,r))} \leq C \left(\frac{r^\theta}{r} \right)^{s} \Bigg\} 
\end{align*}
and, provided $|F| > 0$, the \textit{lower spectrum} of $\mu$ is given by
\begin{align*}
\lowspec \mu = \sup \Bigg\{ s \geq 0 \mid  \exists C>0 \ : \ \forall \ 0<r< \vert F \vert \  :  \ \forall x \in F \  :  \  \frac{\mu(B(x,r^\theta))}{\mu(B(x,r))} \geq C \left(\frac{r^\theta}{r} \right)^{s} \Bigg\} 
\end{align*}
and otherwise it is 0.  Once again, if one replaces $r^\theta$ in the above definitions with a free scale  $R \in (r,1)$, then one recovers the Assouad and lower dimensions of $\mu$, respectively.  The Assouad and lower dimensions of measures  were introduced in \cite{Ka2}, where they were referred to as the upper and lower regularity dimensions, respectively. It is known (see \cite{FFK} for example) that
\[
\low \mu \leq \lowspec \mu \leq \asospec \mu \leq \aso \mu
\]
 and, if $\mu$ has support equal to a closed set $F$, then 
\[
\lowspec \mu \leq \lowspec F \leq \asospec F \leq \asospec \mu.
\]
 The \textit{upper box dimension} of $\mu$  with support $F$ is given by
\begin{align*}
\ubox \mu = \inf \Big\{ s \mid  \exists C>0 \ : \ \forall \ 0<r< \vert F \vert \  :  \ \forall x \in F \  :  \  \mu(B(x,r)) \geq Cr^{s} \Big\} 
\end{align*}
and the \textit{lower box dimension} of $\mu$ is given by
\begin{align*}
\lbox \mu = \inf \Big\{ s \mid \exists  C>0 \  : \ \forall  r_0>0 \ : \  \exists \ 0<r<r_0 \ : \ \forall \ x\in F \  :  \  \mu(B(x,r)) \geq Cr^{s} \Big\}.  
\end{align*}
If $\ubox \mu = \lbox \mu$, then we refer to the common value as the \textit{box dimension} of $\mu$, denoted by $\boxd \mu$. These definitions of the box dimension of a measure were introduced only recently  in \cite{FFK} where it was also shown that, for $\theta \in (0,1)$,
\[\ubox \mu \leq \asospec \mu \leq \min\left\{\aso \mu, \frac{\ubox \mu}{1-\theta}\right\}\]
and if $\mu$ has compact support $F$, then $\ubox F \leq \ubox \mu$.

\subsection{Kleinian groups and limit sets}
\label{KleinPrelims}
For a more thorough study of hyperbolic geometry and Kleinian groups, we refer the reader to \cite{A, B, M,  mcmullen1}. For $d \geq 1$, we model $(d+1)$-dimensional hyperbolic space using the Poincar\'e ball model 
\[
\mathbb{D}^{d+1} = \{z \in \mathbb{R}^{d+1} \mid \vert z \vert < 1\}
\]
  equipped with the hyperbolic metric $d_{\mathbb{H}}$ 
%defined by \[ds = \frac{2 \vert dz \vert}{1-\vert z \vert^2}\] 
and we call the boundary 
\[
\mathbb{S}^d = \{z \in \mathbb{R}^{d+1} \mid \vert z \vert = 1\}
\]
 the \textit{boundary at infinity} of the space $(\mathbb{D}^{d+1},d_{\mathbb{H}})$. We denote by $\text{Con}(d)$ the group of orientation-preserving isometries of  $(\mathbb{D}^{d+1},d_{\mathbb{H}})$. We will occasionally make use of the upper half space model $\mathbb{H}^{d+1} = \mathbb{R}^d \times (0,\infty)$ equipped with the analogous metric.

We say that a group is \textit{Kleinian} if it is a discrete subgroup of $\text{Con}(d)$,
and given a Kleinian group $\Gamma$, the \textit{limit set} of $\Gamma$ is defined to be  $\lset = \overline{\Gamma(\mathbf{0})} \setminus \Gamma(\mathbf{0})$ where $\mathbf{0} = (0,\dots,0) \in \mathbb{D}^{d+1}$. It is well known that $\lset$ is a compact $\Gamma$-invariant subset of $\mathbb{S}^d$. If $\lset$ contains zero, one or two points, it is said to be \textit{elementary}, and otherwise it is \textit{non-elementary}. In the non-elementary  case, $\lset$ is a perfect set, and often has a complicated fractal structure. We consider \emph{geometrically finite} Kleinian groups.  Roughly speaking, this means that there is a fundamental domain with finitely many sides (we refer the reader to \cite{BO} for further details). We define the \textit{Poincar\'e exponent} of a Kleinian group $\Gamma$ to be 
\[\delta = \inf\left\{s>0 \mid \sum_{g \in \Gamma} e^{-s \hdist{\mathbf{0}}{g(\mathbf{0})}} < \infty \right\}.\]
Due to work of  Patterson and Sullivan \cite{Pa1, S1}, it is known that for a non-elementary geometrically finite Kleinian group $\Gamma$, the Hausdorff dimension of the limit set is equal to $\delta$. It was discovered independently by Bishop and Jones \cite[Corollary 1.5]{BJ} and Stratmann and Urba\'nski \cite[Theorem 3]{SU3} that the box and packing dimensions of the limit set are also equal to $\delta$.  Even in the non-elementary geometrically \emph{infinite} case, $\delta$ is still an important quantity.  In fact it  always gives the Hausdorff dimension of the \emph{radial} limit set, and therefore also provides a lower bound for the Hausdorff dimension of the limit set, see \cite{BJ}.  In general, $\delta$ can be difficult to compute or estimate, but there are various techniques available, see \cite{jenkinson, mcmullen3}.

From now on we only discuss the case of non-elementary geometrically finite $\Gamma$. We write  $\mu$ to denote the associated  \textit{Patterson-Sullivan measure}, which is a measure first constructed by Patterson in \cite{Pa1}.  Strictly speaking there  is a family of (mutually equivalent) Patterson-Sullivan measures.  However, we may fix one for simplicity (and hence talk about \emph{the} Patterson-Sullivan measure since the dimension theory is the same for each measure).  The geometry of $\Gamma$, $\lset$ and $\mu$ are heavily related.  For example,  $\mu$ is a $\delta$-conformal $\Gamma$-ergodic Borel probability measure with support  $\lset$.  Moreover, $\mu$ has Hausdorff, packing and entropy dimension equal to $\delta$, see \cite{SV}.  The limit set is $\Gamma$-invariant in the strong sense that $g(\lset) = \lset$ for all $g \in \Gamma$. However, $\mu$ is only quasi-invariant and $\mu \circ g$ is related to $\mu$ by a geometric transition rule, see \cite[Chapter 14]{borth} for a more detailed exposition of this. 

The Assouad and lower dimensions of $\mu$ and $\lset$  were dealt with in \cite{Fr1}. To state the results, we require some more notation.  Suppose  $\Gamma$   contains at least one parabolic point, and denote by $P \subseteq \lset$ the countable set of parabolic fixed  points. We may fix a standard set of horoballs $\{H_p\}_{p \in P}$ (a horoball $H_p$ is a closed Euclidean ball whose interior lies in $\mathbb{D}^{d+1}$ and is tangent to the boundary $\mathbb{S}^{d}$ at $p$) such that they are pairwise disjoint, do not contain the point $\mathbf{0}$, and have the property that for each $g \in \Gamma$ and $p \in P$, we have $g(H_p) = H_{g(p)}$, see \cite{SU3, SV}.

We note that, for any $p \in P$, the stabiliser of $p$ denoted by $\text{Stab}(p)$ cannot contain any loxodromic elements, as this would violate the discreteness of $\Gamma$. We denote by $k(p)$ the maximal rank of a free abelian subgroup of $\text{Stab}(p)$, which must be generated by $k(p)$ parabolic elements which all fix $p$, and call this the \textit{rank} of $p$. We write 
\begin{align*}
\kmin &= \min\{k(p) \mid p \in P\}\\
\kmax &= \max\{k(p) \mid p \in P\}.
\end{align*}
 It was proven in \cite{S1} that $\delta > \kmax/2$. In \cite{Fr1}, the following was proven:
\begin{thm} 
\label{fraserthm}
Let $\Gamma$ be a non-elementary geometrically finite Kleinian group. Then
\begin{align*}
\normalfont{\aso} \lset &= \max\{\delta,k_{\max}\}\\
\normalfont{\low} \lset &= \min\{\delta,k_{\min}\}\\
\normalfont{\aso} \mu &= \max\{2\delta-k_{\min},k_{\max}\}\\
\normalfont{\low} \mu &= \min\{2\delta-k_{\max},k_{\min}\}.
\end{align*}
\end{thm}

We will rely on Stratmann and Velani's \textit{global measure formula} \cite{SV} which gives a formula for the measure of any ball centred in the limit set up to uniform constants. More precisely, given $z \in \lset$ and $T>0$, we define $z_T \in \mathbb{D}^{d+1}$ to be the point on the geodesic ray joining $\mathbf{0}$ and $z$ which is hyperbolic distance $T$ from $\mathbf{0}$. We write $S(z,T) \subset \mathbb{S}^d$ to denote the \textit{shadow at infinity} of the $d$-dimensional hyperplane passing through $z_T$ which is normal to the geodesic joining $\mathbf{0}$ and $z$. The global measure formula states that 
\begin{equation}
\label{GlobThm}
 \mu(S(z,T)) \approx e^{-T\delta}e^{-\rho(z,T)(\delta-k(z,T))}
\end{equation}
where $k(z,T) = k(p)$ if $z_T \in H_p$ for some $p \in P$ and 0 otherwise, and \[\rho(z,T) = \inf\{\hdist{z_T}{y} \mid y \notin H_p \}\] if $z_T \in H_p$ for some $p \in P$ and 0 otherwise.  Basic hyperbolic geometry shows that $S(z,T)$ is a Euclidean ball centred at $z$ with radius comparable to $e^{-T}$, and so an immediate consequence of (\ref{GlobThm}) is the following.
\begin{thm}[Global Measure Formula]
\label{Global}
Let $z \in \lset$, $T>0$. Then we have
\[\muball{z}{T} \approx  e^{-T\delta}e^{-\rho(z,T)(\delta-k(z,T))}.\]
\end{thm}
An easy consequence of Theorem \ref{Global} is that if $\Gamma$ contains no parabolic points, then \[\aso \lset = \low \lset = \aso \mu = \low \mu = \boxd \mu= \delta,\] 
and
\[
\asospec \lset = \asospec \mu = \lowspec \lset = \lowspec \mu  = \delta
\]
for all $\theta \in (0,1)$.  Therefore, we  assume  throughout that $\Gamma$ contains at least one parabolic point.

\section{Results}\label{KleinResults}
We assume throughout that  $\Gamma < \normalfont{\text{Con}}(d)$ is a non-elementary geometrically finite Kleinian group containing at least one parabolic element, and write $\lset$ to denote the associated limit set and $\mu$ to denote the associated Patterson-Sullivan measure.    Our first result gives formulae for the Assouad and lower spectra of $\mu$, as well as the box dimension of $\mu$.
\begin{thm}\label{specmu} Let $\theta \in (0,1).$ Then

\emph{i)}  $\normalfont{\boxd} \mu = \max\{\delta,2\delta-\kmin\}$.

\emph{ii)} If $\delta < \kmin$, then 
\[\normalfont{\asospec} \mu = \delta + \min\left\{1,\frac{\theta}{1-\theta}\right\}(\kmax-\delta),\]

if $\kmin \leq \delta < (\kmin+\kmax)/2$, then  
\[\normalfont{\asospec} \mu = 2\delta -\kmin + \min\left\{1,\frac{\theta}{1-\theta}\right\}(\kmin+\kmax-2\delta)\]

and if $\delta \geq (\kmin+\kmax)/2$, then $\normalfont{\asospec} \mu = 2\delta-\kmin.$

\emph{iii)} If $\delta > \kmax$, then 
\[\normalfont{\lowspec} \mu = \delta - \min\left\{1,\frac{\theta}{1-\theta}\right\}(\delta-\kmin),\]

if $(\kmin+\kmax)/2 < \delta \leq \kmax$, then 
\[\normalfont{\lowspec} \mu = 2\delta-\kmax - \min\left\{1,\frac{\theta}{1-\theta}\right\}(2\delta-\kmin-\kmax)\]

and if $\delta \leq (\kmin+\kmax)/2$, then  $\normalfont{\lowspec} \mu  = 2\delta-\kmax.$
\end{thm}
We prove Theorem \ref{specmu} in Sections \ref{BoxMu} - \ref{LowspecMu}.  It is perhaps noteworthy that  $\kmin$ and $\kmax$ show up \emph{simultaneously} in the formulae for $\asospec \mu$ and $\lowspec \mu$  in the `intermediate range' of  $\delta$.  This does not happen for $\aso \mu$ and $\low \mu$ (see Theorem \ref{fraserthm}) and is indicative of subtle interplay between horoballs of different rank  detected by the Assouad and lower spectra.   The next theorem provides formulae for the Assouad and lower spectra of $\lset$.

\begin{thm}\label{specset} Let $\theta \in (0,1).$ 

\emph{i)} If  $\delta < \kmax$, then 
\[\normalfont{\asospec} \lset = \delta + \min\left\{1,\frac{\theta}{1-\theta}\right\}(\kmax-\delta)\]

and if $\delta \geq \kmax$, then  $\normalfont{\asospec} \lset = \delta.$

\emph{ii)} If $\delta \leq \kmin$, then  $\normalfont{\lowspec} \lset = \delta$, and if $\delta > \kmin$, then 
\[\normalfont{\lowspec} \lset =  \delta - \min\left\{1,\frac{\theta}{1-\theta}\right\}(\delta-\kmin).\]
\end{thm}
We prove Theorem \ref{specset} in Sections \ref{AsospecLimit} and \ref{LowspecLimit}.

\section{Proofs}\label{KleinProofs}
\subsection{Preliminaries}\label{KleinPre}
As many of the proofs in the Kleinian setting are reliant on horoballs, we first  establish various estimates describing the geometry of horoballs, including the `escape functions' $\rho(z,T)$. We start with a simple lemma. One should think of the circle involved as a 2-dimensional slice of a horoball.
\begin{lem}[Circle Lemma]\label{Circle}
Let $R>0$, and consider a circle centred at $(0,R)$ with radius $R$, parametrised by $x(\theta) = R\emph{sin}\theta$ and $y(\theta)=R(1-\emph{cos}\theta)$ $(0 \leq \theta < 2\pi)$. 
For sufficiently small $\theta$, we have
\begin{align}
\frac{\sqrt{Ry(\theta)}}{2} \leq x(\theta) &\leq 2\sqrt{Ry(\theta)} \label{taylor1}\\ 
 \frac{(x(\theta))^2}{4R} \leq y(\theta) &\leq \frac{4(x(\theta))^2}{R}. \label{taylor2} 
\end{align}
\end{lem}

\begin{proof}
By Taylor's Theorem we have $R\theta/2 \leq x(\theta) \leq R\theta$   and $R\theta^2/4 \leq y(\theta) \leq R\theta^2 $ for sufficiently small $\theta$.  Combining these estimates gives the result.
\end{proof}

We also require the following lemma to easily estimate the `escape function' at a parabolic fixed point.
\begin{lem}[Parabolic Centre Lemma]\label{ParaCentre}
Let $p \in \lset$ be a parabolic fixed point with associated standard horoball $H_p$. Then we have $\rho(p,T) \sim T$ as $T \rightarrow \infty$, and for sufficiently large $T>0$ we have $k(p,T) = k(p)$.
\end{lem}
\begin{proof} Let $p_S$ be the `tip' of the horoball $H_p$, in other words the point on the horoball $H_p$ which lies on the geodesic joining \textbf{0} and $p$. It is obvious that $p_T \in H_p \iff T \geq S$, so for sufficiently large $T$ we have $k(p,T)=k(p)$. Also note that
\begin{align*}
1  \geq \frac{\rho(p,T)}{T} 
= \frac{\hdist{p_T}{p_S}}{T}
= \frac{T-S}{T}  \rightarrow 1 
\end{align*}
as $T \rightarrow \infty$,  as required.
\end{proof}

The formulae for $\asospec \mu$ and $\lowspec \mu$  sometimes involve both $\kmin$ and $\kmax$. Consequently,  we   need to consider two standard horoballs, where we drag one towards  the other. One could think of the following lemma as saying that the images of this horoball `fill in the gaps' under the fixed horoball. 

\begin{lem}[Horoball Radius Lemma]\label{HoroRadius}
Let $p,p' \in \lset$ be two parabolic fixed points with associated standard horoballs $H_p$ and $H_{p'}$ respectively, and let $f$ be a parabolic element which fixes $p$. Then for sufficiently large $n$,
\begin{align*}
\vert f^n(p') - p \vert \approx \frac{1}{n} \qquad \text{and} \qquad   \vert f^n(H_{p'}) \vert = \vert H_{f^n(p')} \vert \approx \frac{1}{n^2}.
\end{align*}
\end{lem}

\begin{proof}
By considering an inverted copy of $\mathbb{Z}$ converging to $p$, one 
sees that $\vert f^n(p') - p \vert \approx {1}/{n}$, so we need only prove that $ \vert H_{f^n(p')} \vert \approx {1}/{n^2}$. Note that clearly $\vert H_{f^n(p')} \vert \lesssim {1}/{n^2}$, as otherwise the horoballs would eventually overlap with $H_p$, using Lemma \ref{Circle} and the fact that $\vert f^n(p') - p \vert \approx \frac{1}{n}$, contradicting the fact that our standard set of horoballs is chosen to be pairwise disjoint. For the lower bound, consider a point $u \neq p'$ on $H_{p'}$, and its images under the action of $f$. Let $v$ denote the `shadow at infinity' of $u$ (see Figure \ref{twohoro}), and note that for sufficiently large $n$, $\vert f^n(v) - p \vert \approx \frac{1}{n}$.  As $ f^n(u)$ lies on a horoball with base point $p$, we can use Lemma \ref{Circle} to deduce that $\vert f^n(u) - f^n(v) \vert \approx {1}/{n^2}$ (see Figure \ref{dashhoro}), and therefore $\vert H_{f^n(p')} \vert \gtrsim {1}/{n^2}$, as required.
\end{proof}

\begin{figure}[H]
\centering
\begin{tikzpicture}[scale=0.6]
\draw (5,5) circle (5cm);
\draw (5,2) circle (2cm);
\draw (2.27,6.83) circle (1.7cm);
\draw (1.24,3.42) circle (0.9cm);
\draw (1.87,1.7) circle (0.45cm);
\draw (3.15,0.52) circle (0.15cm);
\draw[dashed] (5,0) arc (270:144:3.23cm);
\node at (5,-0.4) {\large$p$};
\node at (0.7,8.2) {\large$p'$};
\node at (2.6,4.8) {\large$u$};

\draw[dashed] (2.37,5.13) -- (0.03, 5.27) node[left] {\large$v$};

\filldraw[black] (2.37,5.13) circle (2pt) ;
\filldraw[black] (1.81,2.74) circle (2pt) ;
\filldraw[black] (2.28,1.5) circle (2pt) ;
\filldraw[black] (3.3,0.47) circle (2pt) ;
\coordinate (A) at (5.2,0) {};
\coordinate (B) at (5.2,0.2) {};
\coordinate (C) at (5,0.2) {};
\draw (0.015,5.03) -- (0.22,5.015);
\draw (0.22,5.015) -- (0.25,5.25);

\end{tikzpicture}
\caption{An illustration showing the points $u$ and $v$. The dashed arc shows the horoball along which the point $u$ is pulled under the action of $f$.}
\label{twohoro}
\end{figure}
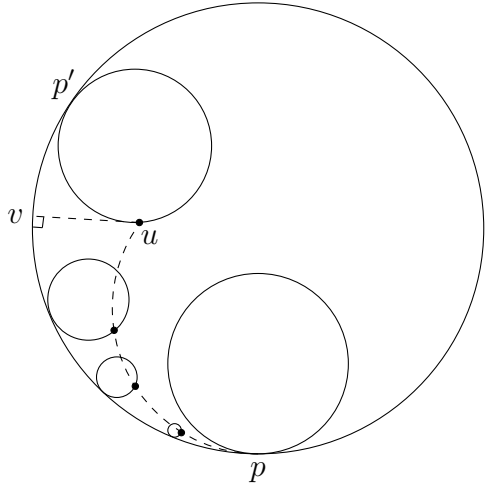
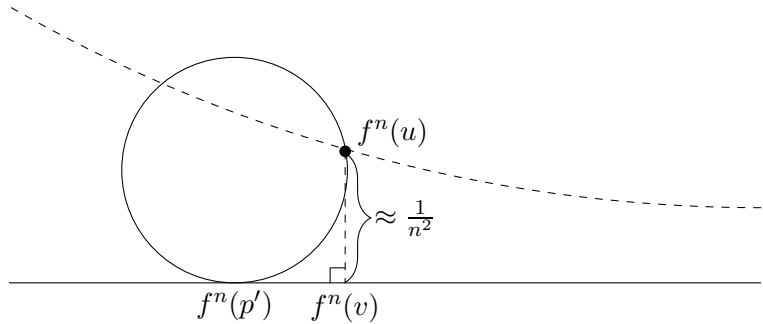
\begin{figure}[H]
\centering
\begin{tikzpicture}
\draw (0,0) -- (10,0);
\draw (3,1.5) circle (1.5cm);
\draw[dashed] (10,1) arc (270:240:20cm);
\draw[dashed] (4.47,1.75) -- (4.47,0) node[below] {$f^n(v)$};
\filldraw[black] (4.47,1.75) circle (2pt) ;
\node at (5.1,2) {$f^n(u)$};
\node at (3,-0.3) {$f^n(p')$};
\draw [decorate,decoration={brace,amplitude=10pt},xshift=-4pt,yshift=0pt]
(4.6,01.73) -- (4.6,0)node [black,midway,xshift=23pt] {
$\approx \large\frac{1}{n^2}$};
\coordinate (A) at (4.27,0) {};
\coordinate (B) at (4.27,0.2) {};
\coordinate (C) at (4.47,0.2) {};
\draw (A) -- (B);
\draw (B) -- (C);
\end{tikzpicture}
\caption{Applying Lemma \ref{Circle}, considering the dashed horoball.}
\label{dashhoro}
\end{figure}

The following lemma is also necessary for estimating hyperbolic distance, essentially saying that if $z,u \in \lset$ are sufficiently close, then $\hdist{z_T}{u_T}$ can be easily estimated.

\begin{lem}\label{Cross}
Let $z,u \in \lset$ and $T>0$ be large. If $ \size{z-u} \approx e^{-T}$, then $\hdist{z_T}{u_T} \approx 1.$ (This should be read as the implicit constants in the conclusion depend on the implicit constants in the assumption.)
\end{lem}

\begin{proof}
This is immediate from the cross ratio formula for hyperbolic distance.    Briefly, given points $P$ and $Q$ in the interior of the Poincar\'e disk, draw a geodesic between them which, upon extension, intersects the boundary of the disc at points $A$ and $B$ such that $A$ is closer to $P$ than $Q$ (see Figure \ref{crossratio}). Then we have
\begin{equation*}
\hdist{P}{Q} =  \text{log} \frac{\vert AQ \vert \vert BP \vert}{\vert AP \vert \vert BQ \vert}.  
\end{equation*}
Applying  this  formula  with $P=z_T$ and $Q=u_T$ yields the result.
\end{proof}
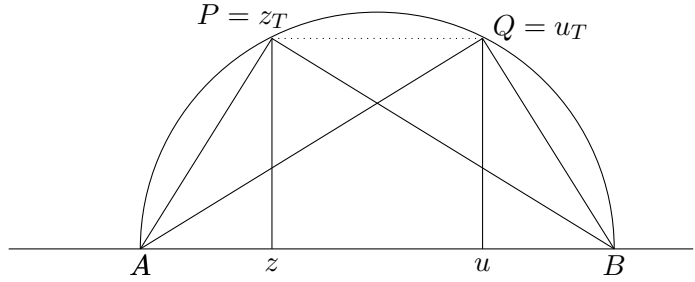
\begin{figure}[H]
\centering
\begin{tikzpicture}[scale=0.7]
\draw (11.5,0) arc (0:180:4.5cm);
\node at (2.5,-0.3) {$A$};
\node at (9,4.2) {$\qquad \qquad  Q=u_T$};
\node at (11.5,-0.3) {$B$};
\node at (5,-0.3) {$z$};
\node at (9,-0.3) {$u$};
\node at (2.5,-0.3) {$A$};
\draw (2.5,0) -- (5,4);
\draw (2.5,0) -- (9,4);
\draw (11.5,0) -- (9,4);
\draw (11.5,0) -- (5,4);
\draw (0,0) -- (13,0);
\draw (5,0) -- (5,4) node[above] {$P=z_T\qquad $};
\draw[dotted] (5,4) -- (9,4);
\draw (9,4) -- (9,0) ;
\end{tikzpicture}
\caption{The cross ratio visualised.}
\label{crossratio}
\end{figure}
 Given a standard horoball $H_p$ and $\lambda \in (0,1]$, we call $\lambda H_p$ the \textit{squeezed horoball} which still has base point $p$, but has Euclidean diameter scaled by a factor of $\lambda$, i.e. $\vert \lambda H_p \vert = \lambda \vert H_p \vert$. We also write $\Pi : \mathbb{D}^{d+1}\setminus\{\textbf{0}\} \rightarrow \mathbb{S}^d$ to denote the projection defined by choosing $\Pi(z) \in \mathbb{S}^d$ such that $\textbf{0}, z$, and $\Pi(z)$ are collinear. Given $A \subset \mathbb{D}^{d+1}\setminus\{\textbf{0}\}$, we call $\Pi(A)$ the \textit{shadow at infinity} of $A$, and note that for a horoball $H_p$, $\Pi(H_p)$ is a Euclidean ball with $\vert \Pi(H_p) \vert \approx \vert H_p \vert$.  We  require the following  lemma due to Stratmann and Velani \cite[Corollary 3.5]{SV} regarding squeezed horoballs.
\begin{lem}\label{Squeeze}
Let $H_p$ be a standard horoball for some $p \in P$, and let $\lambda \in (0,1]$. Then
\begin{equation*}
\normalfont{\mu}(\Pi(\lambda H_p)) \approx \lambda^{2\delta-k(p)} \vert H_p \vert^{\delta}.
\end{equation*}
\end{lem}
We also require the following lemma due to Fraser \cite[Lemma 5.2]{Fr1}, which allows us to count horoballs of certain sizes.
\begin{lem}\label{CountHoro}
Let $z \in \lset$ and $T>t>0$. For $t$ sufficiently large, we have
\begin{equation*}
\sum_{\substack{p \in P \cap B(z,e^{-t}) \\ e^{-t} > \vert H_p \vert \geq e^{-T}}} \vert H_p \vert^{\delta} \lesssim (T-t) \normalfont{\muball{z}{t}}
\end{equation*}
where $P$ is the set of parabolic fixed points contained in $\lset$. 
\end{lem}

\subsection{The box dimension of $\mu$}\label{BoxMu}
\subsubsection{Upper bound}
 We show $\ubox \mu \leq \max\{\delta, 2\delta-\kmin\}.$ Let $z \in \lset$ and $T>0$. We have
\begin{align*}
\muball{z}{T} \nonumber &\gtrsim e^{-T\delta}e^{-\rho(z,T)(\delta-k(z,T))}  \nonumber  
\geq e^{-T\delta}\left(e^{-T}\right)^{\max\{0, \delta-\kmin\}}  
= \left(e^{-T}\right)^{\max\{\delta, 2\delta-\kmin\}} 
\end{align*}
which proves $\ubox\mu \leq \max\{\delta, 2\delta-\kmin\}$.
\subsubsection{Lower bound}
 We show $\lbox \mu \geq \max\{\delta, 2\delta-\kmin\}.$ Note that we have $\lbox \mu \geq \lbox \lset = \delta$, so 
it suffices to prove that $\lbox \mu \geq 2\delta-\kmin$, and therefore we may assume that $\delta \geq \kmin$. Let $p$ be a parabolic fixed point such that $k(p) = \kmin$, and let $\epsilon \in (0,1)$. By Lemma \ref{ParaCentre}, it follows that for sufficiently large $T>0$, we have
\begin{align*}
\muball{p}{T} \nonumber & \lesssim e^{-T\delta}e^{-\rho(p,T)(\delta-k(p))} \nonumber 
 \leq e^{-T\delta}e^{-(1-\epsilon)T(\delta-\kmin)} 
 = \left(e^{-T}\right)^{\delta+(1-\epsilon)(\delta-\kmin)} 
\end{align*}
which proves $\lbox\mu \geq 2\delta-\kmin-\epsilon(\delta-\kmin)$, and letting $\epsilon \rightarrow 0$ proves the lower bound.

\subsection{The Assouad spectrum of $\mu$}\label{AsospecMu}

\subsubsection{When $\delta < \kmin$}\label{AsospecMu1} 
 \textit{Lower bound}: We show \[\asospec\mu \geq \delta + \min\left\{1,\frac{\theta}{1-\theta}\right\}(\kmax-\delta).\]
Let $\theta \in (0,1)$, let $p \in \lset$ be a parabolic fixed point such that $k(p) = \kmax$, $f$ be a parabolic element fixing $p$, and $n \in \mathn$ be very large. Choose $p \neq z_0 \in \lset$, and let $z=f^n(z_0)$, noting that $z \rightarrow p$ as $n \rightarrow \infty$. We assume $n$ is large enough to ensure that the geodesic joining \textbf{0} and $z$ intersects $H_p$, and choose $T>0$ to be the larger of two values such that $z_T$ lies on the boundary of $H_p$. 

We now restrict our attention to the hyperplane $H(p,z,z_T)$ restricted to $\mathbb{D}^{d+1}$. Define $v$ to be the point on $H_p \cap   H(p,z,z_T)$ such that $v$ lies on the quarter circle with centre $z$ and Euclidean radius $e^{-T\theta}$ (see Figure \ref{diag} below). We also consider 2 additional points $u$ and $w$, where $u$ is the `shadow at infinity' of $v$ and $w=u_{T\theta}$.

\begin{figure}[H]
\centering
\begin{tikzpicture}[scale=0.9]
\draw (0,2.35) arc (230:310:10cm);
\draw (0,0) -- (13,0) node[right,below] {$\mathbb{S}^d$};
\draw (0,0) -- (6.5,0) node[below]{$p$};
\draw (8,4.5) -- (8,0) node[below]{$z$};
\node at (7.7,0.4) {$z_T$};
\node at (7.6,4) {$z_{T \theta}$};
\node at (12,1.4) {$v$};
\draw (12,0) arc (0:90:4cm);
\draw[dotted] (8,4) -- (11.7,4) node[right] {$w$};
\draw[dotted] (11.7,4) -- (11.7,1.4);
\draw (11.7,1.5) -- (11.7,0) node[below] {$u$};
\coordinate (A) at (11.5,0) {};
\coordinate (B) at (11.5,0.2) {};
\coordinate (C) at (11.7,0.2) {};
\draw (A) -- (B);
\draw (B) -- (C);
\draw [decorate,decoration={brace,mirror,amplitude=10pt}]
(8,0) -- (8,4)node [black,midway,xshift=30pt] {\footnotesize
$\approx e^{-T\theta}$};
\end{tikzpicture}
\caption{An overview of the horoball $H_p$ along with our chosen points. We wish to find a lower bound for $\hdist{z_{T\theta}}{v}$.}

\label{diag}

\end{figure}
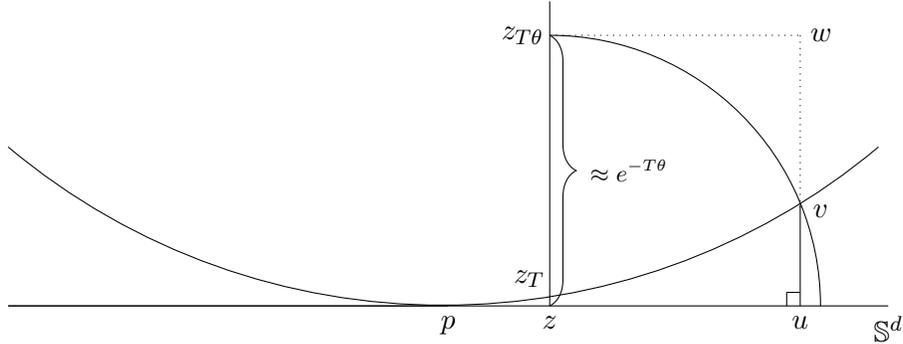
Our goal is to bound $\hdist{v}{w}$ from below and $\hdist{w}{z_{T \theta}}$ from above. Consider the Euclidean distance between $z$ and $p$. Note that $z_T$ lies on the horoball $H_p$, so using (\ref{taylor1}) we have, for sufficiently large $n$,
\begin{equation*}
\vert z - p \vert \lesssim \sqrt{\vert z - z_T \vert} \lesssim e^{-\frac{T}{2}}.
\end{equation*}
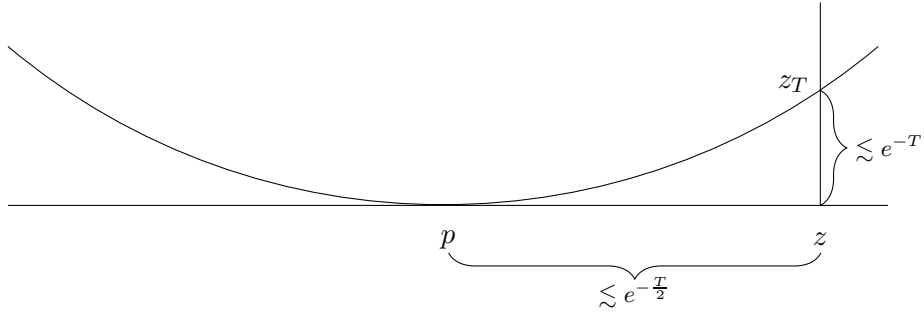
\begin{figure}[H]
\centering
\begin{tikzpicture}[scale=0.9]
\draw (0,0) -- (13,0);
\draw (0,2.35) arc (230:310:10cm);
\node at (6.5,-0.5) {$p$};
\node at (12,-0.5) {$z$};
\node at (11.6,1.8) {$z_T$};
\draw (12,0) -- (12,3);
\draw [decorate,decoration={brace,mirror,amplitude=10pt},xshift=-4pt,yshift=0pt]
(6.65,-0.7) -- (12.15,-0.7)node [black,midway,yshift=-15pt] {\footnotesize
$\lesssim e^{-\frac{T}{2}}$};
\draw [decorate,decoration={brace,amplitude=10pt}]
(12,1.7) -- (12,0)node [black,midway,xshift=26pt] {\footnotesize
$\lesssim e^{-T}$};
\end{tikzpicture}
\caption{Bounding $\vert z - p \vert$ from above using Lemma \ref{Circle}.}
\end{figure}
Also note that clearly $\vert z-u \vert \lesssim e^{-T\theta}$. We can apply (\ref{taylor2}), this time considering the point $v$, and so for sufficiently large $n$, we obtain
\begin{align*}
\vert u - v \vert & \lesssim \left(e^{-\frac{T}{2}}+e^{-T\theta}\right)^2
 \lesssim e^{-\min\{2T\theta,T\}}.
\end{align*} 
\begin{figure}[H]
\centering
\begin{tikzpicture}[scale=0.9]
\draw (0,0) -- (13,0);
\draw (0,2.35) arc (230:310:10cm);
\node at (6.5,-0.5) {$p$};
\node at (12,-0.5) {$u$};
\node at (11.6,1.8) {$v$};
\node at (8,-0.5) {$z$};
\node at (7.7,0.4) {$z_T$};
\draw (8,0) -- (8,3);
\draw (12,0) -- (12,3);
\draw [decorate,decoration={brace,mirror,amplitude=5pt},xshift=-4pt,yshift=0pt]
(6.65,-0.7) -- (12.15,-0.7)node [black,midway,yshift=-17pt] {\footnotesize
$\lesssim e^{-\frac{T}{2}}+e^{-T\theta}$};
\end{tikzpicture}
\caption{Bounding $\vert u - v \vert$ from above, again using Lemma \ref{Circle}.}
\end{figure}
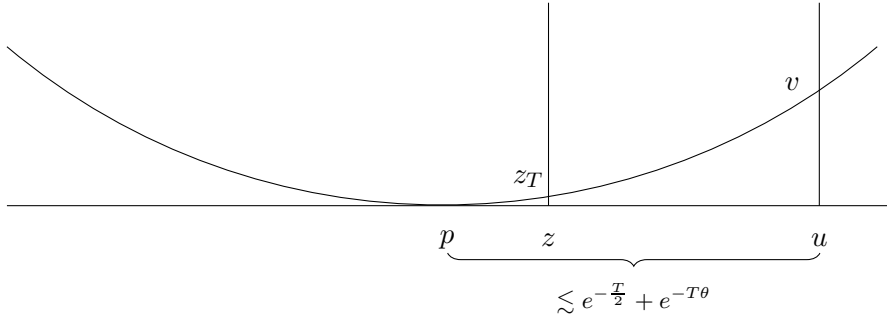
This gives us
\begin{align*}
\hdist{v}{w} &= \text{log}\frac{\vert z-z_{T\theta} \vert}{\vert u-v \vert}  \geq \text{log}\frac{e^{-T\theta}/C_1}{C_2 e^{-\min\{2T\theta,T\}}}
= \min\{T(1-\theta),T\theta\}-\text{log}(C_1C_2) \nonumber
\end{align*}
for some constants $C_1, C_2$. We can apply Lemma \ref{Cross} to deduce that $\hdist{z_{T\theta}}{w} \leq C_3 $ for some constant $C_3$. Therefore, by the triangle inequality
\begin{align}\label{triangle}
\rho(z,T\theta) & = \hdist{z_{T\theta}}{v}   \geq  \hdist{w}{v} - \hdist{z_{T\theta}}{w} 
 \geq \min\{T(1-\theta),T\theta\}-\text{log}(C_1C_2)-C_3. 
\end{align}
By Theorem \ref{Global}, we have
\begin{align*}
\frac{\muball{z}{T\theta}}{\muball{z}{T}}  \approx \frac{e^{-\delta T\theta}}{e^{-\delta T}} \frac{e^{\rho(z,T\theta)(\kmax - \delta)}}{1} 
&\gtrsim \left(e^{T(1-\theta)}\right)^{\delta} e^{\min\{T(1-\theta),T\theta\}(\kmax-\delta)} \qquad  \text{by} \ (\ref{triangle})  \\ 
&=\left(e^{T(1-\theta)}\right)^{\delta} \left(e^{T(1-\theta)}\right)^{\min\left\{1,\frac{\theta}{1-\theta}\right\}(\kmax-\delta)} \\ 
&= \left(e^{T(1-\theta)}\right)^{ \delta + \min\left\{1,\frac{\theta}{1-\theta}\right\}(\kmax-\delta)} \nonumber
\end{align*}
which gives \[\asospec\mu \geq \delta + \min\left\{1,\frac{\theta}{1-\theta}\right\}(\kmax-\delta)\] as required.

 \textit{Upper bound}: We show \[\asospec\mu \leq \delta + \min\left\{1,\frac{\theta}{1-\theta}\right\}(\kmax-\delta).\]
Let $z \in \lset$, $T>0$. Note that we have
\begin{equation*}
\asospec \mu \leq \aso \mu = \kmax
\end{equation*} so we assume that $\theta \in (0,{1}/{2})$. Then by Theorem \ref{Global}, we have
\begin{align*}
\frac{\muball{z}{T\theta}}{\muball{z}{T}}  \approx \frac{e^{-\delta T\theta}}{e^{-\delta T}}\frac{e^{-\rho(z,T\theta)(\delta - k(z,T\theta))}}{e^{-\rho(z,T)(\delta -k(z,T))}} 
&\lesssim \left(e^{T(1-\theta)}\right)^{\delta} e^{\rho(z,T\theta)(\kmax-\delta)}\\ 
&\leq \left(e^{T(1-\theta)}\right)^{\delta} e^{T\theta(\kmax-\delta)}\\ 
&= \left(e^{T(1-\theta)}\right)^{\delta + \frac{\theta}{1-\theta}(\kmax-\delta)}
\end{align*}
which gives \[\asospec\mu \leq \delta + \frac{\theta}{1-\theta}(\kmax-\delta)\] as required.

\subsubsection{When $\kmin \leq \delta < {(\kmin+\kmax)}/{2}$}\label{AsospecMu2} 
 \textit{Lower bound}: We show \[\asospec\mu \geq 2\delta-\kmin + \min\left\{1,\frac{\theta}{1-\theta}\right\}(\kmin+\kmax-2\delta).\]
Let $\theta \in (0,{1}/{2})$, let $p,p' \in \lset$ be parabolic fixed points such that $k(p) = \kmax$ and $k(p') = \kmin$, and let $f$ be a parabolic element fixing $p$. Let $n$ be a large positive integer and let $z=f^n(p')$. By Lemma \ref{HoroRadius}, we note that for sufficiently large $n$ we may choose $T$ such that $k(z,T) = \kmin$, $k(z,T\theta) = \kmax$, and $\vert z - p \vert = e^{-T\theta}$ (see Figure \ref{twoballs}).

We can make use of Lemma \ref{Cross} to deduce that
\begin{equation*}
\hdist{z_{T\theta}}{p_{T\theta}} \leq C_1
\end{equation*}
for some constant $C_1$. Also, by Lemma \ref{ParaCentre}, given $\epsilon \in (0,1)$ we have that $\rho(p,T\theta) \geq (1-\epsilon)T\theta$ for sufficiently large $n$. This gives
\begin{align}\label{paracen}
\rho(z,T\theta) \geq \rho(p,T\theta) - C_1 
\geq (1-\epsilon)T\theta - C_1.
\end{align}
Finally, we note that as $\vert z - p \vert =  e^{-T\theta}$, by Lemma \ref{Circle} we have $\vert H_z \vert \approx e^{-2T\theta}$ (see Figure \ref{twoballs2}), which implies that
\begin{equation}\label{taylor3}
\rho(z,T)  \geq \text{log}\frac{e^{-2T\theta}/C_2}{C_3e^{-T}} = T(1-2\theta) - \text{log}(C_2C_3)
\end{equation}
for some constants $C_2, C_3$.
\begin{figure}[H]
\centering
\begin{tikzpicture}[scale=0.9]
\draw (0,0) -- (13,0);
\draw (1,3.96) arc (200:340:6cm);
\node at (6.61,-0.5) {$p$};
\node at (4,-0.5) {$z$};
\draw(4,0) -- (4,5);
\filldraw[black] (4,3) circle (2pt) ;
\filldraw[black] (4,0.3) circle (2pt) ;
\draw[dashed] (4,0.3) -- (3.3,0.3) node[left] {$z_T$};
\node at (3.5,3) {$z_{T\theta}$};
\draw (4,0.3) circle (0.3cm);
\draw (6.61,0) -- (6.61,5);
\filldraw[black] (6.61,3) circle (2pt);
\node at (7.15,3) {$p_{T\theta}$};
\draw[dashed] (4,3) -- (6.61,3);
\coordinate (A) at (4.2,0) {};
\coordinate (B) at (4.2,0.2) {};
\coordinate (C) at (4,0.2) {};
\coordinate (D) at (6.81,0) {};
\coordinate (E) at (6.81,0.2) {};
\coordinate (F) at (6.61,0.2) {};
\draw (A) -- (B);
\draw (B) -- (C);
\draw (D) -- (E);
\draw (E) -- (F);
\draw [decorate,decoration={brace,mirror,amplitude=10pt},xshift=0pt,yshift=0pt]
(4,-0.7) -- (6.61,-0.7)node [black,midway,xshift=5pt, yshift=-17pt] {\footnotesize
$e^{-T\theta}$};
\end{tikzpicture}
\caption{Making use of Lemma \ref{HoroRadius} so we can choose our desired $T$.}
\label{twoballs}
\end{figure}
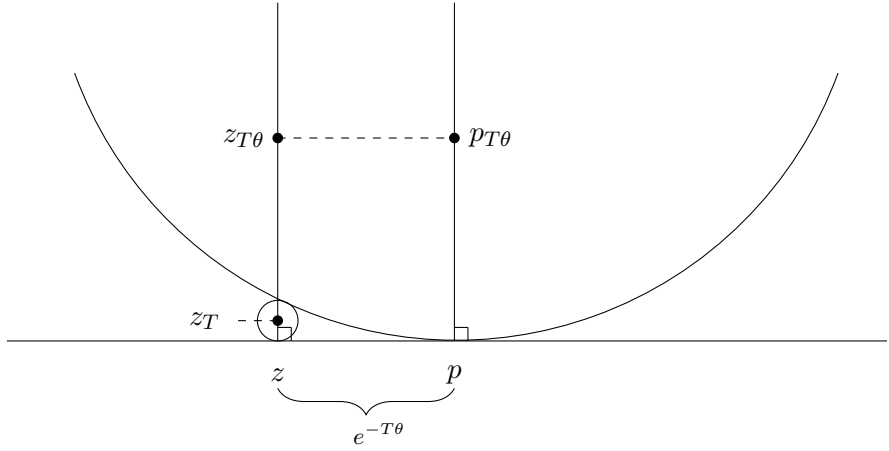

\begin{figure}[H]
\centering
\begin{tikzpicture}[scale=0.9]
\draw (0,0) -- (13,0);
\draw (6.5,2) circle (2cm);
\draw (6.5,5) -- (6.5,0) node[below] {$z$};
\filldraw[black] (6.5,2) circle (2pt);
\draw (13,2.5) arc (260:247:50cm);
\node at (7,2) {$z_T$};
\draw [decorate,decoration={brace,mirror,amplitude=10pt},xshift=0pt,yshift=0pt]
(6.5,0) -- (6.5,2)node [black,midway,xshift=25pt] {\footnotesize
$\approx e^{-T}$};
\draw [decorate,decoration={brace,amplitude=10pt},xshift=0pt,yshift=0pt]
(4.5,0) -- (4.5,4)node [black,midway,xshift=-28pt] {\footnotesize
$\approx e^{-2T\theta}$};
\coordinate (A) at (6.3,0) {};
\coordinate (B) at (6.3,0.2) {};
\coordinate (C) at (6.5,0.2) {};
\draw (A) -- (B);
\draw (B) -- (C);
\end{tikzpicture}
\caption{Calculating $\rho(z,T)$ using Lemma \ref{Circle}.}
\label{twoballs2}
\end{figure}
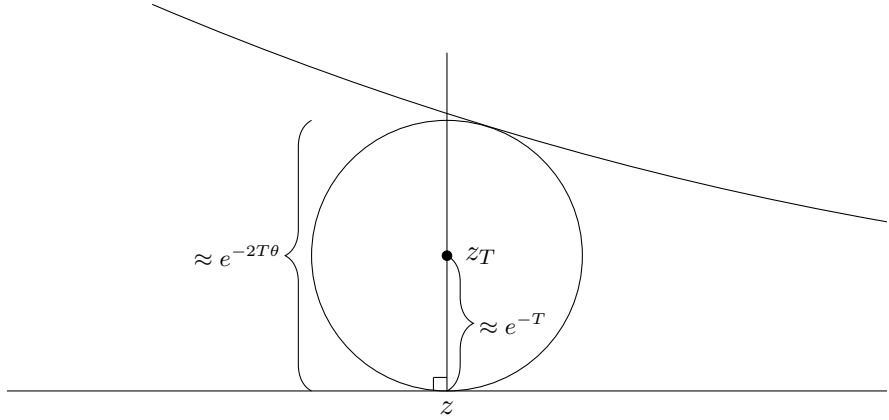
Applying (\ref{Global}), we get
\begin{align*}
\frac{\muball{z}{T\theta}}{\muball{z}{T}}  &\approx \frac{e^{-\delta T\theta}}{e^{-\delta T}} \frac{e^{\rho(z,T\theta)(\kmax - \delta)}}{e^{-\rho(z,T)(\delta - \kmin)}} \\
&\gtrsim \left(e^{T(1-\theta)}\right)^{\delta} e^{(1-\epsilon)T\theta(\kmax-\delta)+(1-2\theta)T(\delta-\kmin)} \  \text{by} \ (\ref{paracen}) \ \text{and} \ (\ref{taylor3})\\ 
&=\left(e^{T(1-\theta)}\right)^{\delta} e^{T(1-\theta)\delta + T(1-\theta) \kmin + T\theta(\kmin+\kmax-2\delta)-\epsilon T\theta(\kmax-\delta)} \\ 
&\geq \left(e^{T(1-\theta)}\right)^{2\delta - \kmin + \frac{\theta}{1-\theta}(\kmin+\kmax-2\delta)-\epsilon\frac{\theta}{1-\theta}(\kmax-\delta)} 
\end{align*}
which proves \[\asospec \mu \geq 2\delta - \kmin + \frac{\theta}{1-\theta}(\kmin+\kmax-2\delta)-\epsilon\frac{\theta}{1-\theta}(\kmax-\delta)\] and letting $\epsilon \rightarrow 0$ proves the desired lower bound.

The case when $\theta \geq {1}/{2}$ follows identically to the lower bound in Section \ref{AsospecMu1}, so we omit the details.

 \textit{Upper bound}: We show \[\asospec\mu \leq 2\delta-\kmin + \min\left\{1,\frac{\theta}{1-\theta}\right\}(\kmin+\kmax-2\delta).\]
Let $z \in \lset$, $T>0$, $\theta \in (0,{1}/{2})$ (the case when $\theta \geq {1}/{2}$ follows similarly to the upper bound in Section \ref{AsospecMu1}). If $z_{T}$ and $z_{T\theta}$ do not lie in a common standard horoball, we may use the fact that 
\begin{equation}\label{escape1}
\rho(z,T) \leq T(1-\theta) - \rho(z,T\theta).
\end{equation}
We have by Theorem \ref{Global}
\begin{align*}
\frac{\muball{z}{T\theta}}{\muball{z}{T}}  \approx \frac{e^{-\delta T\theta}}{e^{-\delta T}}\frac{e^{-\rho(z,T\theta)(\delta - k(z,T\theta))}}{e^{-\rho(z,T)(\delta -k(z,T))}}
&\leq \frac{e^{-\delta T\theta}}{e^{-\delta T}}\frac{e^{-\rho(z,T\theta)(\delta - \kmax)}}{e^{-\rho(z,T)(\delta -\kmin)}} \\
&= \left(e^{T(1-\theta)}\right)^{\delta} e^{\rho(z,T\theta)(\kmax-\delta) + \rho(z,T)(\delta -\kmin)}\\
&\leq  \left(e^{T(1-\theta)}\right)^{\delta} e^{T\theta(\kmin+\kmax-2\delta) + T(1-\theta)(\delta -\kmin)} \ \text{by} \ (\ref{escape1})\\
&\leq  \left(e^{T(1-\theta)}\right)^{2\delta-\kmin + \frac{\theta}{1-\theta}(\kmin+\kmax-2\delta)}.
\end{align*}
If $z_{T}$ and $z_{T\theta}$ do lie in a common standard horoball $H_p$ and $\delta \geq k(p)$, then we use the inequality
\begin{equation}\label{escape2}
\vert \rho(z,T) - \rho(z,T\theta) \vert \leq T(1-\theta).
\end{equation}
Then we have
\begin{align*}
\frac{\muball{z}{T\theta}}{\muball{z}{T}}  \approx \frac{e^{-\delta T\theta}}{e^{-\delta T}}\frac{e^{-\rho(z,T\theta)(\delta - k(z,T\theta))}}{e^{-\rho(z,T)(\delta -k(z,T))}} 
&= \left(e^{T(1-\theta)}\right)^{\delta} e^{(\rho(z,T)-\rho(z,T\theta))(\delta -k(p))}\\
&\leq \left(e^{T(1-\theta)}\right)^{\delta} e^{T(1-\theta)(\delta -\kmin)} \quad  \text{by} \ (\ref{escape2})\\
&=  \left(e^{T(1-\theta)}\right)^{2\delta-\kmin}\\
&\leq  \left(e^{T(1-\theta)}\right)^{2\delta-\kmin + \frac{\theta}{1-\theta}(\kmin+\kmax-2\delta)}
\end{align*}
if $\rho(z,T)-\rho(z,T\theta) \geq 0$, and otherwise
\begin{align*}
\frac{\muball{z}{T\theta}}{\muball{z}{T}}  
\approx \left(e^{T(1-\theta)}\right)^{\delta} e^{(\rho(z,T)-\rho(z,T\theta))(\delta -k(p))} 
&\leq \left(e^{T(1-\theta)}\right)^{\delta} \\
&\leq  \left(e^{T(1-\theta)}\right)^{2\delta-\kmin + \frac{\theta}{1-\theta}(\kmin+\kmax-2\delta)}.
\end{align*}
If $\delta < k(p)$, then we use
\begin{equation}\label{escape3}
\rho(z,T\theta) - \rho(z,T) \leq T\theta. 
\end{equation}
By Theorem \ref{Global},
\begin{align*}
\frac{\muball{z}{T\theta}}{\muball{z}{T}}  
 \approx \left(e^{T(1-\theta)}\right)^{\delta} e^{(\rho(z,T)-\rho(z,T\theta))(\delta -k(p))}
&\leq \left(e^{T(1-\theta)}\right)^{\delta} e^{T\theta(k(p)-\delta)} \quad  \text{by} \ (\ref{escape3})\\
&\leq  \left(e^{T(1-\theta)}\right)^{\delta + \frac{\theta}{1-\theta}(\kmax-\delta)}\\
&\leq  \left(e^{T(1-\theta)}\right)^{2\delta-\kmin + \frac{\theta}{1-\theta}(\kmin+\kmax-2\delta)}.
\end{align*}
In all cases, we have \[\asospec\mu \leq 2\delta-\kmin + \frac{\theta}{1-\theta}(\kmin+\kmax-2\delta)\] as required.

\subsubsection{When $\delta \geq  {(\kmin+\kmax)}/{2}$}\label{AsospecMu3}
 We show $\asospec \mu = 2\delta-\kmin.$ This follows easily, since  
\[
2 \delta -\kmin = \ubox \mu \leq \asospec \mu \leq \aso \mu = 2 \delta-\kmin
\]
 and so $\asospec \mu = 2\delta-\kmin$, as required.

\subsection{The lower spectrum of $\mu$}\label{LowspecMu}
The proofs of the bounds for the lower spectrum follow similarly to the Assouad spectrum, and so we only sketch the arguments. 

\subsubsection{When $\delta > \kmax$}\label{LowspecMu1} 
 \textit{Upper bound}: We show \[\lowspec\mu \leq  \delta - \min\left\{1,\frac{\theta}{1-\theta}\right\}(\delta-\kmin).\]
Let $\theta \in (0,1)$, $p \in \lset$ be a parabolic fixed point such that $k(p) = \kmin$, $f$ be a parabolic element fixing $p$, and $n \in \mathn$ be very large. Choose $p \neq z_0 \in \lset$, and let $z=f^n(z_0)$. We choose $T>0$ such that $z_T$ is the 'exit point' from $H_p$. Identically to the lower bound in Section \ref{AsospecMu1}, we have that
\begin{equation*}
\rho(z,T\theta)   \geq \min\{T(1-\theta),T\theta\}-C 
\end{equation*}
for some constant $C$. Applying Theorem \ref{Global}, we have
\begin{align*}
\frac{\muball{z}{T\theta}}{\muball{z}{T}}  \approx \frac{e^{-\delta T\theta}}{e^{-\delta T}} \frac{e^{-\rho(z,T\theta)(\delta-\kmin)}}{1} 
&\lesssim \left(e^{T(1-\theta)}\right)^{\delta} e^{-\min\{T(1-\theta),T\theta\}(\delta-\kmin)}\\ 
&=\left(e^{T(1-\theta)}\right)^{\delta}\left(e^{T(1-\theta)}\right)^{-\min\left\{1,\frac{\theta}{1-\theta}\right\}(\delta-\kmin)} \\ 
&= \left(e^{T(1-\theta)}\right)^{ \delta - \min\left\{1,\frac{\theta}{1-\theta}\right\}(\delta-\kmin)} 
\end{align*}
which gives \[\lowspec\mu \leq  \delta - \min\left\{1,\frac{\theta}{1-\theta}\right\}(\delta-\kmin)\] as required.

 \textit{Lower bound}: We show \[\lowspec\mu \geq  \delta - \min\left\{1,\frac{\theta}{1-\theta}\right\}(\delta-\kmin).\]
Let $z \in \lset$, $T>0$. Note that $\lowspec \mu \geq \low \mu = \kmin$ so we assume $\theta \in (0,{1}/{2})$. Then by Theorem \ref{Global}, we have
\begin{align*}
\frac{\muball{z}{T\theta}}{\muball{z}{T}}  \approx \frac{e^{-\delta T\theta}}{e^{-\delta T}}\frac{e^{-\rho(z,T\theta)(\delta - k(z,T\theta))}}{e^{-\rho(z,T)(\delta -k(z,T))}}
&\gtrsim \left(e^{T(1-\theta)}\right)^{\delta} e^{-\rho(z,T\theta)(\delta-\kmin)} \\
&\geq \left(e^{T(1-\theta)}\right)^{\delta} e^{-T\theta(\delta-\kmin)}\\ 
&= \left(e^{T(1-\theta)}\right)^{\delta - \frac{\theta}{1-\theta}(\delta-\kmin)}
\end{align*}

which gives \[\lowspec\mu \geq  \delta - \frac{\theta}{1-\theta}(\delta-\kmin)\] as required.

\subsubsection{When ${(\kmin+\kmax)}/{2} < \delta \leq \kmax$}\label{LowspecMu2} 
 \textit{Upper bound}: We show \[\lowspec\mu \leq  2\delta-\kmax - \min\left\{1,\frac{\theta}{1-\theta}\right\}(2\delta-\kmin-\kmax).\] 
Similarly to the lower bound in Section \ref{AsospecMu2}, we only need to deal with the case when $\theta \in (0,{1}/{2})$. Let $p,p' \in \lset$ be parabolic fixed points such that $k(p) = \kmin$ and $k(p') = \kmax$, and let $f$ be a parabolic element fixing $p$. Let $n$ be a large integer and let $z = f^n(p')$. Again, by Lemma \ref{HoroRadius}, for sufficiently large $n$ we may choose $T$ such that $k(z,T)=\kmax$ and $k(z,T\theta)=\kmin$. We may argue in the same manner as the lower bound in Section \ref{AsospecMu2} to show that, for sufficiently large $n$,
\begin{align*}
\rho(z,T\theta) \geq (1-\epsilon)T\theta - C_1 \quad 
\text{and} \qquad  \rho(z,T) \geq T(1-2\theta) - C_2
\end{align*}
for some $\epsilon \in (0,1)$ and some constants $C_1,C_2$. Applying Theorem \ref{Global} gives
\begin{align*}
\frac{\muball{z}{T\theta}}{\muball{z}{T}}  &\approx \frac{e^{-\delta T\theta}}{e^{-\delta T}} \frac{e^{\rho(z,T\theta)(\kmin-\delta)}}{e^{-\rho(z,T)(\delta - \kmax)}}\\ 
&\lesssim \left(e^{T(1-\theta)}\right)^{\delta} e^{(1-\epsilon)T\theta(\kmin-\delta)+(1-2\theta)T(\delta-\kmax)}\\ 
&=\left(e^{T(1-\theta)}\right)^{\delta} e^{T(1-\theta)\delta - T(1-\theta)\kmax  - T\theta(2\delta-\kmin-\kmax)+\epsilon T\theta(\delta-\kmin)} \\ 
&= \left(e^{T(1-\theta)}\right)^{2\delta-\kmax - \frac{\theta}{1-\theta}(2\delta-\kmin-\kmax)+\epsilon\frac{\theta}{1-\theta}(\delta-\kmin)} 
\end{align*}
which gives \[\lowspec \mu \leq 2\delta-\kmax - \frac{\theta}{1-\theta}(2\delta-\kmin-\kmax)+\epsilon\frac{\theta}{1-\theta}(\delta-\kmin)\] and letting $\epsilon \rightarrow 0$ proves the desired upper bound.

 \textit{Lower bound}: We show \[\lowspec\mu \geq  2\delta-\kmax - \min\left\{1,\frac{\theta}{1-\theta}\right\}(2\delta-\kmin-\kmax).\]
Let $z \in \lset$, $T>0$, $\theta \in (0,{1}/{2})$ (the case when $\theta \geq {1}/{2}$ follows similarly to the lower bound in Section \ref{LowspecMu1}). Suppose $z_T$ and $z_{T\theta}$ do not lie in a common standard horoball.
Applying Theorem \ref{Global} gives
\begin{align*}
\frac{\muball{z}{T\theta}}{\muball{z}{T}}  \approx \frac{e^{-\delta T\theta}}{e^{-\delta T}}\frac{e^{-\rho(z,T\theta)(\delta - k(z,T\theta))}}{e^{-\rho(z,T)(\delta -k(z,T))}}
&\geq \frac{e^{-\delta T\theta}}{e^{-\delta T}}\frac{e^{-\rho(z,T\theta)(\delta - \kmin)}}{e^{-\rho(z,T)(\delta -\kmax)}} \\
&= \left(e^{T(1-\theta)}\right)^{\delta} e^{\rho(z,T\theta)(\kmin - \delta) + \rho(z,T)(\delta -\kmax)}\\
&\geq  \left(e^{T(1-\theta)}\right)^{\delta} e^{T\theta(\kmax+\kmin-2\delta) + T(1-\theta)(\delta -\kmin)} \quad  \text{by} \ (\ref{escape1})\\
&\geq  \left(e^{T(1-\theta)}\right)^{2\delta-\kmin - \frac{\theta}{1-\theta}(2\delta-\kmin-\kmax)}.
\end{align*}
If $z_T$ and $z_{T\theta}$ do lie in a common standard horoball $H_p$ and $\delta \leq k(p)$, then we have
\begin{align*}
\frac{\muball{z}{T\theta}}{\muball{z}{T}}  \approx \frac{e^{-\delta T\theta}}{e^{-\delta T}}\frac{e^{-\rho(z,T\theta)(\delta - k(z,T\theta))}}{e^{-\rho(z,T)(\delta -k(z,T))}} 
&= \left(e^{T(1-\theta)}\right)^{\delta} e^{(\rho(z,T)-\rho(z,T\theta))(\delta -k(p))}\\
&\geq \left(e^{T(1-\theta)}\right)^{\delta} e^{T(1-\theta)(\delta -\kmax)} \quad  \text{by} \ (\ref{escape2}) \\
&=  \left(e^{T(1-\theta)}\right)^{2\delta-\kmax}\\
&\geq  \left(e^{T(1-\theta)}\right)^{2\delta-\kmax  - \frac{\theta}{1-\theta}(2\delta-\kmin-\kmax)}
\end{align*}
if $\rho(z,T)-\rho(z,T\theta) \geq 0$, and otherwise
\begin{align*}
\frac{\muball{z}{T\theta}}{\muball{z}{T}}
\approx \left(e^{T(1-\theta)}\right)^{\delta} e^{(\rho(z,T)-\rho(z,T\theta))(\delta -k(p))}
&\geq \left(e^{T(1-\theta)}\right)^{\delta} \\
&\geq  \left(e^{T(1-\theta)}\right)^{2\delta-\kmax  - \frac{\theta}{1-\theta}(2\delta-\kmin-\kmax)}.
\end{align*}
If $\delta > k(p)$, then
\begin{align*}
\frac{\muball{z}{T\theta}}{\muball{z}{T}} 
\approx \left(e^{T(1-\theta)}\right)^{\delta} e^{(\rho(z,T)-\rho(z,T\theta))(\delta -k(p))}
&\geq \left(e^{T(1-\theta)}\right)^{\delta} e^{-T\theta(\delta-k(p))} \quad  \text{by} \ (\ref{escape3})\\
&\geq  \left(e^{T(1-\theta)}\right)^{\delta - \frac{\theta}{1-\theta}(\delta-\kmin)}\\
&\geq  \left(e^{T(1-\theta)}\right)^{2\delta-\kmax - \frac{\theta}{1-\theta}(2\delta-\kmin-\kmax)}.
\end{align*}
In all cases, we have \[\lowspec\mu \geq  2\delta-\kmax - \frac{\theta}{1-\theta}(2\delta-\kmin-\kmax)\] as required.

\subsubsection{When $\delta \leq {(\kmin+\kmax)}/{2}$}\label{LowspecMu3} 
 We show $\lowspec \mu = 2\delta-\kmax.$ Let $p,p' \in \lset$ be two parabolic fixed points such that $k(p') = \kmax$, let $f$ be a parabolic element fixing $p$ and let $n$ be a large positive integer. Let $z = f^n(p')$ and let $\epsilon \in (0,1)$. Note that for sufficiently large $n$ and Lemma \ref{ParaCentre}, we may choose $T$ such that
\begin{equation*}
\rho(z,T) \geq (1-\epsilon)T \ \text{and} \ \rho(z,T\theta)=0.
\end{equation*}
\begin{figure}[H]
\centering
\begin{tikzpicture}[scale=0.75]
\draw (0,0) -- (13,0);
\draw (6.5,2) circle (2cm);
\draw (6.5,5) -- (6.5,0) node[below] {$z$};
\filldraw[black] (6.5,0.5) circle (2pt);
\filldraw[black] (6.5,4) circle (2pt);
\node at (7,0.5) {$z_T$};
\node at (6.0,4.25) {$z_{T\theta}$};
\coordinate (A) at (6.3,0) {};
\coordinate (B) at (6.3,0.2) {};
\coordinate (C) at (6.5,0.2) {};
\draw (A) -- (B);
\draw (B) -- (C);
\end{tikzpicture}
\caption{Choosing the appropriate $T$.}
\end{figure}
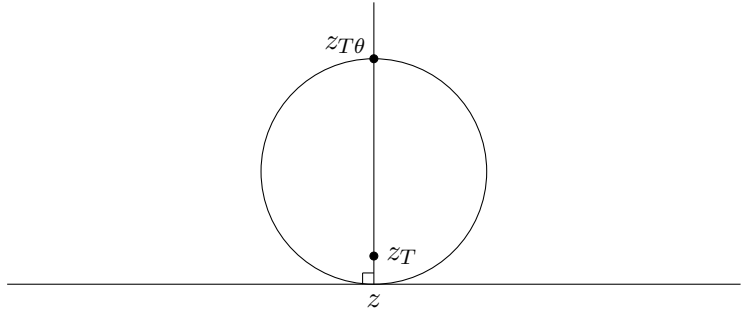
By Theorem \ref{Global}, we have 
\begin{align*}
\frac{\muball{z}{T\theta}}{\muball{z}{T}}  \approx \frac{e^{-\delta T\theta}}{e^{-\delta T}}\frac{1}{e^{-\rho(z,T)(\delta - \kmax)}} 
&\lesssim \left(e^{T(1-\theta)}\right)^{\delta} e^{(1-\epsilon)T(1-\theta)(\delta-\kmax)}\\
&= \left(e^{T(1-\theta)}\right)^{2\delta-\kmax-\epsilon(\delta-\kmax)}
\end{align*}
which proves $\lowspec \mu \leq 2\delta-\kmax-\epsilon(\delta-\kmax)$  and letting $\epsilon \rightarrow 0$ proves the upper bound.

Recall that $\low \mu = \min\{2\delta-\kmax,\kmin\}$, so when $ \delta \leq {(\kmin+\kmax)}/{2}$, we have 
\begin{equation*}
2\delta-\kmax = \low \mu   \leq \lowspec \mu \leq 2\delta-\kmax
\end{equation*}
so $\lowspec \mu = 2\delta-\kmax$, as required.

\subsection{The Assouad spectrum of $\lset$}\label{AsospecLimit}

\subsubsection {When $\delta \leq \kmin$}\label{AsospecLimit1} 
 We show \[\asospec \lset = \delta + \min\left\{1,\frac{\theta}{1-\theta}\right\}(\kmax-\delta).\]
As 
\begin{equation}\label{ubound}
\asospec \lset \leq \asospec \mu = \delta + \min\left\{1,\frac{\theta}{1-\theta}\right\}(\kmax-\delta)
\end{equation}
 when $\delta \leq \kmin$, we need only prove the lower bound. To obtain this, we make use of the following result (see \cite[Theorem 3.4.8]{Fr2}).
\begin{prop}\label{rho}
Let $F \subseteq \mathr^n$ and suppose that \[\rho =\inf \normalfont{\{\theta \in (0,1) \mid \asospec F = \aso F\}}\] exists and $\rho \in (0,1)$ and $ \normalfont{\low F = \ubox F}$. Then for $\theta \in (0,\rho)$,
\begin{equation*}
 \normalfont{\asospec F \geq \ubox F + \frac{(1-\rho)\theta}{(1-\theta)\rho}(\aso F-\ubox F)}.
\end{equation*}
\end{prop}
We note that if $\delta \leq \kmin$, then certainly \[\low \lset = \min\{\kmin,\delta\} = \delta = \ubox \lset.\] We now show that $\rho = {1}/{2}$.
By (\ref{ubound}), we have $\rho \geq {1}/{2}$, so we need only show that $\rho \leq {1}/{2}$. To do this, we recall a result from \cite{Fr1}, which states that using the orbit of a free abelian subgroup of the stabiliser of some parabolic fixed point $p$ with $k(p) = \kmax$, we can find a bi-Lipschitz image of an inverted $\mathz^{\kmax}$ lattice inside $\lset$ in $\mathbb{S}^d$, which gives us
\begin{equation}\label{lattice}
\asospec \lset \geq \asospec \frac{1}{\mathz^{\kmax}} = \min\left\{\kmax,\frac{\kmax}{2(1-\theta)}\right\}.
\end{equation}
The final equality giving the Assouad spectrum of $1/\mathbb{Z}^{\kmax}$ is a straightforward calculation which we omit, but  note that this can be derived along similar lines (and giving the same formula) as the treatment of the product set $(1/\mathbb{N})^{\kmax}$ given in \cite[Prop 4.5, Cor 6.4]{FYu}. This proves $\rho \leq {1}/{2}$, as required. 
Therefore, by Proposition \ref{rho}, we have for $\theta \in (0,{1}/{2})$
\begin{align*}
\asospec \lset \geq  \ubox \lset + \frac{(1-\rho)\theta}{(1-\theta)\rho}(\aso \lset-\ubox \lset) 
 = \delta + \frac{\theta}{1-\theta}(\kmax-\delta)
\end{align*}
as required.

\subsubsection{When $\kmin < \delta < \kmax$}\label{AsospecLimit2}
 \textit{Upper bound}: We show \[\asospec \lset \leq \delta + \min\left\{1,\frac{\theta}{1-\theta}\right\}(\kmax-\delta).\]
The argument for the upper bound for the Assouad spectrum is nearly identical to the argument for the upper bound for the Assouad dimension given in \cite{Fr1}, and so we omit any part of the argument which does not improve upon the bounds provided in the paper, in particular when the Assouad dimension is bounded above by $\delta$. We also note that
\[\asospec \lset \leq \aso \lset = \kmax\]
so we may assume that $\theta \in (0,1/2)$. 
Let $z \in \lset$, $\epsilon > 0$, and $T$ be sufficiently large such that $T(1-\theta) \geq \max\{\epsilon^{-1},\text{log}10\}$. Let $\{x_i\}_{i \in X}$ be a centred $e^{-T}$-packing of $B(z,e^{-T\theta}) \cap \lset$ of maximal cardinality, in other words $x_i \in B(z,e^{-T\theta}) \cap \lset$ for all $i \in X$ and $\vert x_i - x_j \vert > 2e^{-T}$ for $i \neq j$. Decompose $X$ as follows
\begin{equation*}
X = X_0 + X_1 + \bigcup_{n=2}^{\infty} X_n
\end{equation*}
where
\begin{align*}
X_0 &= \{i \in X \mid (x_i)_T \in H_p \ \text{with} \ \vert H_p \vert \geq 10e^{-T\theta}\}\\
X_1 &= \{i \in X\setminus X_0 \mid \rho(x_i,T) \leq \epsilon T(1-\theta)\}\\
X_n &= \{i \in X \setminus (X_0 \cup X_1) \mid n-1 < \rho(x_i,T) \leq n\}.
\end{align*}
Our goal is to bound from above the cardinalities of $X_0$, $X_1$ and $X_n$ separately.

We start with $X_0$, which we may assume is non-empty as we are trying to bound from above. We note that if there exists some $p \in P$ with $\vert H_p \vert \geq 10e^{-T\theta}$ and $H_p \cap \left(\cup_{i \in X} (x_i)_T\right) \neq \emptyset$, then this $p$ must be unique, i.e. if $H_p \cap \left(\cup_{i \in X}(x_i)_T\right) \neq \emptyset$ and $H_{p'} \cap \left(\cup_{i \in X}(x_i)_T\right) \neq \emptyset$ for some $p,p' \in P$ with $\vert H_p \vert, \vert H_{p'} \vert  \geq 10e^{-T\theta}$, then $H_p$ and $H_{p'}$ could not be disjoint. This means we can choose $p \in P$ such that $(x_i)_T \in H_p$ for all $i \in X_0$, and also note that this forces $z_{T\theta} \in H_p$.

If $\delta \leq k(p)$, then by Theorem \ref{Global}
\begin{align*}
e^{-T\theta \delta}e^{-\rho(z,T\theta)(\delta-k(p))}  \gtrsim \muball{z}{T\theta}  
\geq \mu (\cup_{i \in X_0} B(x_i,e^{-T}))
\gtrsim \vert X_0 \vert \min_{i \in X_0} (e^{-T})^{\delta} e^{-\rho(x_i,T)(\delta-k(p))}
\end{align*}
where the second inequality follows from the fact that $\{x_i\}_{i \in X_0}$ is an $e^{-T}$ packing. Therefore
\begin{align}
\vert X_0 \vert   \lesssim \max_{i \in X_0} \left(e^{T(1-\theta)}\right)^{\delta} e^{(\rho(x_i,T)-\rho(z,T\theta))(\delta-k(p))} \nonumber
&\leq \left(e^{T(1-\theta)}\right)^{\delta} e^{-T\theta(\delta-\kmax)} \nonumber \\
&= \left(e^{T(1-\theta)}\right)^{\delta+\frac{\theta}{1-\theta}(\kmax - \delta)} \label{ubound1}.
\end{align}
Note that Fraser \cite{Fr1} estimates using $\rho(x_i,T)-\rho(z,t) \leq T-t+10$, but we can improve this for $t = T\theta$ with $\theta < {1}/{2}$ by using $\rho(x_i,T) \geq 0$ and $\rho(z,T\theta) \leq T\theta$.

If $\delta > k(p)$, then we refer the reader to \cite[4998-5000]{Fr1}, where it is shown that \[\vert X_0 \vert \lesssim \left(e^{T(1-\theta)}\right)^{\delta}.\]
Therefore, we have \[ \vert X_0 \vert \lesssim  \left(e^{T(1-\theta)}\right)^{\delta+\frac{\theta}{1-\theta}(\kmax - \delta)}\] regardless of the relationship between $\delta$ and $k(p)$.

For $X_1$, we have 
\begin{align*}
e^{-T\theta \delta}e^{-\rho(z,T\theta)(\delta-k(z,T\theta))}  \gtrsim \muball{z}{T\theta}  
\geq \mu (\cup_{i \in X_1} B(x_i,e^{-T}))
&\gtrsim \sum_{i \in X_1}  (e^{-T})^{\delta} e^{-\rho(x_i,T)(\delta-k(x_i,T))}\\
&\geq \vert X_1 \vert (e^{-T})^{\delta} e^{-\epsilon T(1-\theta)(\delta-\kmin)}
\end{align*}
by the definition of $X_1$ and the fact that $\delta > \kmin$. Therefore
\begin{align}
\vert X_1 \vert \nonumber \lesssim  \left(e^{T(1-\theta)}\right)^{\delta} e^{\epsilon T(1-\theta)(\delta-\kmin)-\rho(z,T\theta)(\delta-k(z,T\theta))} 
&\leq \left(e^{T(1-\theta)}\right)^{\delta} e^{\epsilon T(1-\theta)(\delta-\kmin)-T\theta(\delta-\kmax)}\nonumber \\
&= \left(e^{T(1-\theta)}\right)^{\delta + \frac{\theta}{1-\theta}(\kmax - \delta)+\epsilon (\delta-\kmin)}\label{ubound2}.
\end{align}

Finally, we consider the sets $X_n$. If $i \in X_n$ for $n \geq 2$, then $\rho(x_i,T) > n-1$, and so $(x_i)_T \in H_p$ for some $p \in P$ with $e^{-T} \leq \vert H_p \vert < 10e^{-T\theta}$. Furthermore, we may note that $B(x_i,e^{-T})$ is contained in the shadow at infinity of the squeezed horoball $2e^{-(n-1)}H_p.$ Also note that as $\vert H_p \vert < 10e^{-T\theta}$, we have $p \in B(z,10e^{-T\theta})$. For each integer $k \in [\kmin,\kmax]$ define
\[X_n^k = \{ i \in X_n \mid k(x_i,T) = k \}.\]
Then
\begin{align*}
\mu \left(\bigcup_{i \in X_n^k} B(x_i,e^{-T}) \right) &\leq \mu \left(\bigcup_{\substack{p \in P \cap B(z,10e^{-T\theta}) \\ 10e^{-T\theta} > \vert H_p \vert \geq e^{-T}\\ k(p)=k}} \Pi\left(2e^{-(n-1)}H_p\right)\right)\\
&\lesssim \sum_{\substack{p \in P \cap B(z,10e^{-T\theta}) \\ 10e^{-T\theta} > \vert H_p \vert \geq e^{-T}\\ k(p)=k}} \mu(\Pi(2e^{-(n-1)}H_p))\\
&\lesssim e^{-n(2\delta-k)} \sum_{\substack{p \in P \cap B(z,10e^{-T\theta}) \\ 10e^{-T\theta} > \vert H_p \vert \geq e^{-T}}} \vert H_p \vert^{\delta} \qquad  \text{by Lemma} \ \ref{Squeeze}\\
&\lesssim e^{-n(2\delta-k)} (T(1-\theta)+ \ \text{log}10) \muball{z}{T\theta}  \\
&\lesssim e^{-n(2\delta-k)} T(1-\theta) e^{-T\theta\delta} e^{-\rho(z,T\theta)(\delta-k(z,T\theta))} \\
&\lesssim e^{-n(2\delta-k)} \epsilon^{-1}n  e^{-T\theta\delta} e^{T\theta(\kmax-\delta)} 
\end{align*}
where the last three inequalities use Theorem \ref{Global}, Lemma \ref{CountHoro} and the fact that $\epsilon T(1-\theta) < \rho(x_i,T) \leq n$ as $i \notin X_1$. On the other hand, using the fact that $\{x_i\}_{i \in X_n^k}$ is an $e^{-T}$ packing, 
\begin{align*}
\mu \left(\bigcup_{i \in X_n^k} B(x_i,e^{-T}) \right) \geq \sum_{i \in X_n^k} \mu(B(x_i,e^{-T})) \gtrsim \vert X_n^k \vert e^{-T\delta} e^{-n(\delta-k)}
\end{align*}
using $n-1 < \rho(x_i,T) \leq n$. Therefore
\begin{align*}
\vert X_n^k \vert  \lesssim \epsilon^{-1}n e^{-n\delta} \left(e^{T(1-\theta)}\right)^{\delta}   e^{T\theta(\kmax-\delta)} 
= \epsilon^{-1}n e^{-n\delta}  \left(e^{T(1-\theta)}\right)^{\delta + \frac{\theta}{1-\theta}(\kmax - \delta)}
\end{align*}
which implies
\begin{align}\label{ubound3}
\vert X_n \vert = \sum_{k=\kmin}^{\kmax} \vert X_n^k \vert
\lesssim \epsilon^{-1}n e^{-n\delta}  \left(e^{T(1-\theta)}\right)^{\delta + \frac{\theta}{1-\theta}(\kmax - \delta)}.
\end{align}
Combining (\ref{ubound1}), (\ref{ubound2}) and (\ref{ubound3}), we have
\begin{align*}
\vert X \vert &= \vert X_0 \vert + \vert X_1 \vert + \sum_{n=2}^{\infty} \vert X_n \vert\\
&\lesssim  \left(e^{T(1-\theta)}\right)^{\delta + \frac{\theta}{1-\theta}(\kmax - \delta)+\epsilon (\delta-\kmin)} +  \sum_{n=2}^{\infty} \epsilon^{-1}n e^{-n\delta}  \left(e^{T(1-\theta)}\right)^{\delta + \frac{\theta}{1-\theta}(\kmax - \delta)}\\
&\lesssim  \left(e^{T(1-\theta)}\right)^{\delta + \frac{\theta}{1-\theta}(\kmax - \delta)+\epsilon (\delta-\kmin)} +  \epsilon^{-1}\left(e^{T(1-\theta)}\right)^{\delta + \frac{\theta}{1-\theta}(\kmax - \delta)}
\end{align*}
which proves \[\asospec \lset \leq \delta + \frac{\theta}{1-\theta}(\kmax - \delta)+\epsilon (\delta-\kmin)\]
and letting $\epsilon \rightarrow 0$ proves the desired upper bound.

 \textit{Lower bound}: We show \[\asospec \lset \geq \delta + \min\left\{1,\frac{\theta}{1-\theta}\right\}(\kmax-\delta).\]
Note that the case when $\theta \geq 1/2$ is an immediate consequence of (\ref{lattice}), so we assume that $\theta \in (0,1/2)$.
In order to derive the lower bound, we let $p \in \lset$ be a parabolic fixed point with $k(p) = \kmax$, then we switch to the upper half-space model $\mathbb{H}^{d+1}$ and assume that $p = \infty$. Recalling the argument from \cite[Section 5.1.1]{Fr1}, we know that there exists some subset of $\lset$ which is bi-Lipschitz equivalent to $\mathbb{Z}^{\kmax}$. Therefore, we may assume that we have a $\mathbb{Z}^{\kmax}$ lattice, denoted by $Z$, as a subset of $\lset$, and then use the fact that the Assouad spectrum is stable under bi-Lipschitz maps.

We now partition this lattice into $\{Z_k\}_{k \in \mathbb{N}}$, where \[Z_k = \{z \in Z \mid 10(k-1) \leq \vert z \vert < 10k\} .\]
We note that 
\begin{equation}\label{ZK1}
\vert Z_k \vert \approx k^{\kmax-1}.
\end{equation}
Let $\phi : \mathbb{H}^{d+1} \rightarrow \mathbb{D}^{d+1}$ denote the Cayley transformation mapping the upper half space model to the Poincar\'e ball model, and consider the family of balls $\{B(z,{1}/{3})\}_{z \in Z}$.  Taking the $\phi$-image of $Z$ yields an inverted lattice and  there are positive constants $C_1$ and $C_2$ such that if
$z \in Z_k$ for some $k$, then
\begin{align*}
\frac{1}{C_1k} \leq \vert \phi(z) - p_1 \vert \leq \frac{C_1}{k} \qquad 
\text{and} \qquad   \frac{1}{C_2k^2} \leq \vert \phi(B(z,{1}/{3})) \vert \leq \frac{C_2}{k^2},
\end{align*}
where $p_1 = \phi(\infty)$.

Let $T>0$. We now choose a constant $C_3$ small enough which satisfies the following:
\begin{itemize}
\item The set of balls $\bigcup_{k \in \mathbb{N}}\bigcup_{z \in Z_k} (B(\phi(z),C_3/{k^2}))$ are pairwise disjoint.
\item If $z \in Z_k$ and $\vert \phi(z) - p_1 \vert < e^{-T\theta}$, then $B(\phi(z),C_3/{k^2}) \subset B(p_1, 2e^{-T\theta})$.
\end{itemize}
This gives us
\begin{equation}\label{covering1}
N_{e^{-T}}(B(p_1,2e^{-T\theta}) \cap \lset) \gtrsim \sum_{\substack{z \in Z_k  \\ {C_3}/{k^2} >  e^{-T} \\ \vert \phi(z) - p_1 \vert < e^{-T\theta}}} N_{e^{-T}}(B(\phi(z),C_3/k^2) \cap \lset).
\end{equation}
We wish to estimate 
\[
N_{e^{-T}}(B(\phi(z),C_3/{k^2}) \cap \lset)
\]
 from below. Let $k \in \mathbb{N}$ satisfy the conditions given in the sum above, and write $t =\text{log}({k^2}/C_3)$, and let $\epsilon \in (0,1)$.  We  will restrict the above sum to ensure that   $T-t \geq \max\{\epsilon^{-1},\text{log}10\}$, that is, we assume further that $k \leq a(\epsilon)e^{T/2} \sqrt{C_3}$ where $a(\epsilon)= \sqrt{\exp(-\max\{\epsilon^{-1},\text{log}10\})}$. 

Let $\{y_i\}_{i \in Y}$ be a centred $e^{-T}$ covering of $B(\phi(z),e^{-t}) \cap \lset$, and decompose $Y$ as $Y = Y_0 \cup Y_1$ where
\begin{align*}
Y_0 &= \{i \in Y \mid (y_i)_T \in H_p \ \text{with} \ \vert H_p \vert \geq 10e^{-t}\}\\
Y_1 &= Y\setminus Y_0.
\end{align*}
We wish to estimate the cardinalities of $Y_0$ and $Y_1$, and we refer the reader to \cite[Section 5.2]{Fr1} for the details regarding the estimates used.

If $i \in Y_0$, then $(y_i)_T \in H_p$ for some unique parabolic fixed point $p$ with $\vert H_p \vert \geq 10e^{-t}$. If $ \delta < k(p)$, it can be shown that
\[\vert Y_0 \vert \gtrsim \left(\frac{e^{-t}}{e^{-T}}\right)^{\delta}\]
and if $\delta \geq k(p)$ then
\begin{align}\label{card1}
\vert Y_0 \vert   \gtrsim \left(\frac{e^{-t}}{e^{-T}}\right)^{\delta} \min_{i \in Y_0} e^{(\rho(y_i,T)-\rho(\phi(z),t))(\delta-k(p))} 
 \gtrsim \left(\frac{e^{-t}}{e^{-T}}\right)^{\delta} e^{-\rho(\phi(z),t)(\delta-k(p))}.
\end{align}
In order to estimate $\rho(\phi(z),t)$, we ensure that $T$ is chosen large enough such that the line joining \textbf{0} and $\phi(z)$ intersects $H_{p_1}$, and let $u>0$ be the larger of two values such that $\phi(z)_u$ lies on the boundary of $H_{p_1}$ (see Figure \ref{fig1}). Then note that \[\vert \phi(z) - \phi(z)_t \vert \approx e^{-t} = C_3/{k^2}\] and also that since $ \vert \phi(z) - p_1 \vert \approx \frac{1}{k}$,  Lemma \ref{Circle} gives $\vert \phi(z) -  \phi(z)_u \vert \approx 1/k^2.$
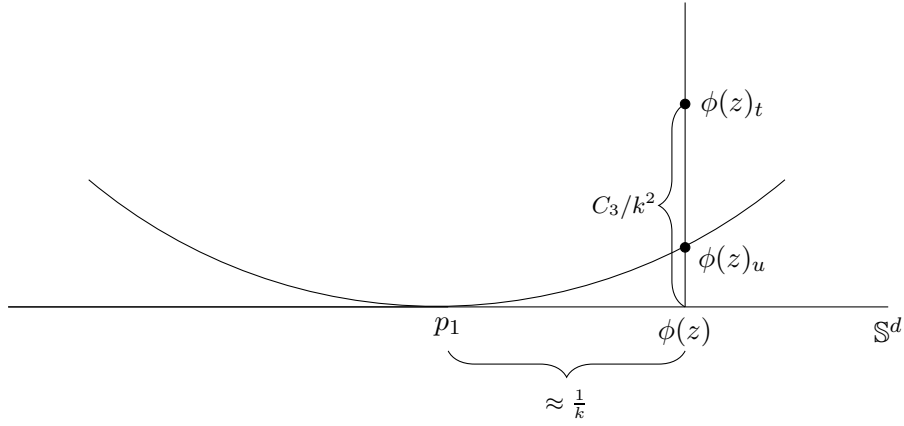
\begin{figure}[H]
\centering
\begin{tikzpicture}[scale=0.9]
\draw (1.19,1.88) arc (230:310:8cm);
\draw (0,0) -- (13,0) node[right,below] {$\mathbb{S}^d$};
\draw (0,0) -- (6.5,0) node[below]{$p_1$};
\draw (10,4.5) -- (10,0) node[below]{$\phi(z)$};
\node at (10.7,3) {$\phi(z)_t$};
\node at (10.7,0.7) {$\phi(z)_u$};
\filldraw[black] (10,3) circle (2pt);
\filldraw[black] (10,0.88) circle (2pt);
\draw [decorate,decoration={brace,mirror,amplitude=10pt},xshift=0pt,yshift=0pt]
(6.5,-0.65) -- (10,-0.65) node [black,midway,yshift=-20pt] {\footnotesize
$\approx \frac{1}{k}$};
\draw [decorate,decoration={brace,amplitude=10pt},xshift=0pt,yshift=0pt]
(10,0) -- (10,3) node [black,midway,xshift=-23pt] {\footnotesize
$C_3/{k^2}$};
\end{tikzpicture}
\caption{Estimating $\rho(\phi(z),t)$ when $t<u$ using Lemma \ref{Circle}.}

\label{fig1}

\end{figure}
Therefore, we have
\begin{align}
\rho(\phi(z),t) &\leq  \text{log}\frac{\vert \phi(z) - \phi(z)_t \vert}{\vert \phi(z) -  \phi(z)_u \vert}
 \leq C_4 \label{rhoc4}
\end{align}
for some constant $C_4$. Therefore, by (\ref{card1})
\begin{align*}
\vert Y_0 \vert &\gtrsim \left(\frac{e^{-t}}{e^{-T}}\right)^{\delta} e^{-C_4(\delta-\kmin)}
\approx  \left(\frac{e^{-t}}{e^{-T}}\right)^{\delta}.
\end{align*}
Similarly, we can show that 
\begin{align*}
\vert Y_1 \vert \gtrsim \left(\frac{e^{-t}}{e^{-T}}\right)^{\delta+\epsilon(\delta-\kmax)} e^{-\rho(\phi(z),t)(\delta-\kmin)} \gtrsim \left(\frac{e^{-t}}{e^{-T}}\right)^{\delta+\epsilon(\delta-\kmax)}.
\end{align*}
Therefore, we have
\begin{align*}
N_{e^{-T}}(B(\phi(z),e^{-t}) \cap \lset) \gtrsim \left(\frac{e^{-t}}{e^{-T}}\right)^{\delta+\epsilon(\delta-\kmax)} 
= \left(\frac{C_3/{k^2}}{e^{-T}}\right)^{\delta+\epsilon(\delta-\kmax)}
\end{align*}
and so by (\ref{covering1})
\begin{align*}
  N_{e^{-T}}(B(p_1,2e^{-T\theta}) \cap \lset)
 &\gtrsim  \sum_{\substack{z \in Z_k  \\ {C_3}/{k^2} >  e^{-T} \\ \vert \phi(z) - p_1 \vert < e^{-T\theta}}}  \left(\frac{C_3/{k^2}}{e^{-T}}\right)^{\delta+\epsilon(\delta-\kmax)} \\
&\geq e^{T(\delta+\epsilon(\delta-\kmax))} \sum_{k = \lceil e^{T\theta}/C_1 \rceil}^{\lfloor a(\epsilon)\sqrt{C_3}e^{T/2} \rfloor} \vert Z_k \vert  \left(C_3/{k^2}\right)^{\delta+\epsilon(\delta-\kmax)} \\
&\gtrsim e^{T(\delta+\epsilon(\delta-\kmax))} \sum_{k = \lceil e^{T\theta}/C_1 \rceil}^{\lfloor a(\epsilon)\sqrt{C_3}e^{T/2} \rfloor} k^{\kmax-1-2\delta-2\epsilon(\delta-\kmax)}  \\
&\gtrsim e^{T(\delta+\epsilon(\delta-\kmax))} \int_{e^{T\theta}/C_1 }^{a(\epsilon) \sqrt{C_3}e^{T/2} } x^{\kmax-1-2\delta-2\epsilon(\delta-\kmax)} dx
\end{align*}
where the second last inequality uses (\ref{ZK1}). We may assume $\epsilon$ is chosen small enough to ensure that $\kmax-2\delta-2\epsilon(\delta-\kmax)<0$ and that $T$ is large enough (depending on $\epsilon$) such that $a(\epsilon) \sqrt{C_3}e^{T/2}>e^{T\theta}/C_1 $. Then the smaller limit of integration provides the dominant term, so we have
\begin{align*}
N_{e^{-T}}(B(p_1,2e^{-T\theta}) \cap \lset)  &\gtrsim  e^{T(\delta+\epsilon(\delta-\kmax))} e^{T\theta(\kmax-2\delta-2\epsilon(\delta-\kmax))}\\
&= \left(e^{T(1-\theta)}\right)^{\delta+\frac{\theta}{1-\theta}(\kmax-\delta)+\epsilon(\delta-\kmax)} e^{-T\theta \epsilon(\delta-\kmax)} \\
&\geq  \left(e^{T(1-\theta)}\right)^{\delta+\frac{\theta}{1-\theta}(\kmax-\delta)+\epsilon(\delta-\kmax)}
\end{align*}
which proves \[\asospec \lset \geq \delta+\frac{\theta}{1-\theta}(\kmax-\delta)+\epsilon(\delta-\kmax)\] and letting $\epsilon \rightarrow 0$ proves the lower bound.

\subsubsection{When $\delta \geq \kmax$}\label{AsospecLimit3}  We show $\asospec \lset = \delta.$ This follows easily, since  when $\delta \geq \kmax$, we have 
\[
\delta = \ubox \lset \leq  \asospec \lset \leq \aso \lset = \delta
\]
 and so $\asospec \lset = \delta$, as required.

\subsection{The lower spectrum of $\lset$}\label{LowspecLimit}

\subsubsection{When $\delta \leq \kmin$}\label{LowspecLimit1}  We show $\lowspec \lset = \delta.$ When $\delta \leq \kmin$, we have 
\[
\delta = \low \lset \leq \lowspec \lset \leq \lbox \lset = \delta
\]
 and so $\lowspec \lset = \delta$, as required.

\subsubsection{When $\kmin < \delta < \kmax$}\label{LowspecLimit2} 
 \textit{Lower bound}: We show \[\lowspec\lset \geq  \delta - \min\left\{1,\frac{\theta}{1-\theta}\right\}(\delta-\kmin).\]
Similarly to the upper bound for the Assouad spectrum of $\lset$, the argument is essentially identical to the one given in \cite{Fr1} for the lower bound for the lower dimension, and so we only include parts of the argument which improve upon the bounds derived in that paper, and we may assume that $\theta \in (0,1/2)$.

Let $z \in \lset$, $\epsilon \in (0,1)$, and $T>0$ with $T(1-\theta) \geq \max\{\epsilon^{-1},\text{log}10\}$. Let $\{y_i\}_{i \in Y}$ be a centred $e^{-T}$-covering of $B(z,e^{-T\theta}) \cap \lset$ of minimal cardinality. Decompose $Y$ as $Y = Y_0 \cup Y_1$
where
\begin{align*}
Y_0 &= \{i \in Y \mid (y_i)_T \in H_p \ \text{with} \ \vert H_p \vert \geq 10e^{-T\theta}\}\\
Y_1 &= Y\setminus Y_0.
\end{align*}
As $\{y_i\}_{i \in Y}$ is a centred $e^{-T}$-covering of $B(z,e^{-T\theta})  \cap \lset$, we have
\begin{align}
\muball{z}{T\theta}  \leq \mu \left(\cup_{i \in Y} B(y_i,e^{-T}) \right)
\leq \mu \left(\cup_{i \in Y_0} B(y_i,e^{-T}) \right) + \mu \left(\cup_{i \in Y_1} B(y_i,e^{-T}) \right)\label{decomp}
\end{align}
so one of the terms in (\ref{decomp}) must be at least $\muball{z}{T\theta}/2$.

Suppose that \[\mu \left(\cup_{i \in Y_0} B(y_i,e^{-T}) \right) \geq \muball{z}{T\theta}/2.\] This means that $\vert Y_0 \vert \neq \emptyset$, and so similar to the Assouad spectrum argument, we can show that there exists a unique $p \in P$ with $\vert H_p \vert \geq 10e^{-T\theta}$ such that $(y_i)_T \in H_p$ for all $i \in Y_0$. Moreover,this forces $z_{T\theta} \in H_p$. 

If $\delta \geq k(p)$, then
\begin{align*}
e^{-T\theta \delta} e^{-\rho(z,T\theta)(\delta - k(p))} \lesssim \muball{z}{T\theta} 
\leq 2\mu \left(\cup_{i \in Y_0} B(y_i,e^{-T}) \right) 
\lesssim \vert Y_0 \vert \max_{i \in Y_0} e^{-T\delta} e^{-\rho(y_i,T)(\delta - k(p))}
\end{align*}
so we have
\begin{align*}
\vert Y_0 \vert  \gtrsim \left(e^{T(1-\theta)}\right)^{\delta} \min_{i \in Y_0} e^{(\rho(y_i,T)-\rho(z,T\theta))(\delta-k(p))} 
\gtrsim \left(e^{T(1-\theta)}\right)^{\delta} e^{-T\theta(\delta-\kmin)} 
= \left(e^{T(1-\theta)}\right)^{\delta-\frac{\theta}{1-\theta}(\delta-\kmin)}.
\end{align*}
If $\delta < k(p)$, we refer the reader to \cite[5003-5005]{Fr1} where it is shown that \[\vert Y_0 \vert \gtrsim \left(e^{T(1-\theta)}\right)^{\delta}.\]
Therefore
\begin{equation}\label{lbound1}
\vert Y_0 \vert \gtrsim \left(e^{T(1-\theta)}\right)^{\delta-\frac{\theta}{1-\theta}(\delta-\kmin)} 
\end{equation}
regardless of the relationship between $\delta$ and $k(p)$.
Now, suppose that \[\mu \left(\cup_{i \in Y_1} B(y_i,e^{-T}) \right) \geq \muball{z}{T\theta}/2.\]
Define \[Y_1^{0} = \{i \in Y_1 \mid \rho(y_i,T) \leq \epsilon T(1-\theta)\}\]
and suppose $i \in Y_1 \setminus Y_1^0$. Then $\rho(y_i,T) > \epsilon T(1-\theta)$ and so similar to the $X_n$ case in the Assouad spectrum argument, we know that $(y_i)_T$ is contained in the squeezed horoball $e^{-\epsilon T(1-\theta)}H_p$ with $e^{-T} \leq \vert H_p \vert < 10e^{-T\theta}$ for some $p \in P$ with $p \in B(z,10e^{-T\theta})$.

We also note that the Euclidean distance from $(y_i)_T$ to $\mathbb{S}^d$ is comparable to $e^{-T}$, and so by Pythagoras' Theorem, see Figure \ref{pythagfig}, we have
\begin{align*}
\vert y_i - p \vert  \lesssim \sqrt{(e^{-\epsilon T(1-\theta)}\vert H_p \vert)^2 - (e^{-\epsilon T(1-\theta)}\vert H_p \vert - e^{-T})^2} \lesssim \sqrt{e^{-\epsilon T(1-\theta)}\vert H_p \vert e^{-T}}.
\end{align*}

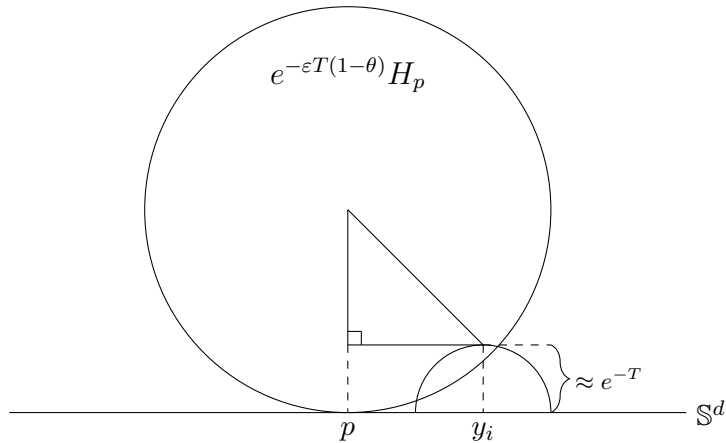
\begin{figure}[H]
\centering
\begin{tikzpicture}[scale=0.9]
\node at (5,5) {\large $e^{-\epsilon T(1-\theta)}H_p$};
\draw (0,0) -- (10,0) node[right] {$\mathbb{S}^d$};
\draw (5,3) -- (7,1);
\draw (7,1) -- (5,1);
\draw (5,1) -- (5,3);
\draw[dashed] (7,1) -- (7,0) node[below] {$y_i$};
\draw[dashed] (5,1) -- (5,0) node[below] {$p$};
\draw[dashed] (7,1) -- (8,1);
\draw (5,3) circle (3cm);
\draw (8,0) arc (0:180:1cm);
\coordinate (A) at (5.2,1) {};
\coordinate (B) at (5.2,1.2) {};
\coordinate (C) at (5,1.2) {};
\draw (A) -- (B);
\draw (B) -- (C);
\draw [decorate,decoration={brace,amplitude=7pt},xshift=0pt,yshift=0pt]
(8,1) -- (8,0)node [black,midway,xshift=22pt] {\footnotesize
$\approx e^{-T}$};
\end{tikzpicture}
\caption{Applying Pythagoras' Theorem using $(y_i)_T$, $p_T$ and the centre of $e^{-\epsilon T(1-\theta)} H_p$.} 

\label{pythagfig}

\end{figure}
Therefore, we have that $B(y_i, e^{-T})$ is contained in the shadow at infinity of the squeezed horoball 
\[C\sqrt{\frac{e^{-\epsilon T(1-\theta)-T}}{\vert H_p \vert}}H_p\]
for some constant $C$. By Lemma \ref{Squeeze} and Lemma \ref{CountHoro}, we have 
\begin{align*}
\mu \left(\cup_{i \in Y_1 \setminus Y_0^1} B(y_i,e^{-T}) \right) &\leq \sum_{\substack{p \in P \cap B(z,10e^{-T\theta}) \\ 10e^{-T\theta} > \vert H_p \vert \geq e^{-T}}} \mu \left( \Pi \left(C\sqrt{\frac{e^{-\epsilon T(1-\theta)-T}}{\vert H_p \vert}}H_p\right)\right)\\
&\approx  \sum_{\substack{p \in P \cap B(z,10e^{-T\theta}) \\ 10e^{-T\theta} > \vert H_p \vert \geq e^{-T}}} \left(\sqrt{\frac{e^{-\epsilon T(1-\theta)-T}}{\vert H_p \vert}}\right)^{2\delta-k(p)} \vert H_p \vert^{\delta}  \\
&\leq e^{-\epsilon T(1-\theta)(\delta-\kmax/2)}  \sum_{\substack{p \in P \cap B(z,10e^{-T\theta}) \\ 10e^{-T\theta} > \vert H_p \vert \geq e^{-T}}} \vert H_p \vert^{\delta}\\
&\lesssim e^{-\epsilon T(1-\theta)(\delta-\kmax/2)}(T(1-\theta)+\text{log}10)\muball{z}{T\theta}.
\end{align*}
As $\delta > \kmax/2$, this means that balls with centres in $Y_1 \setminus Y_1^0$ cannot carry a fixed proportion of $\mu(B(z,e^{-T\theta}))$ for sufficiently large $T$. It follows that
\begin{align*}
\mu \left(\cup_{i \in Y_0^1} B(y_i,e^{-T}) \right)  \approx \mu \left(\cup_{i \in Y_1} B(y_i,e^{-T}) \right) 
\geq \mu(B(z,e^{-T\theta}))/2 
\end{align*}
by our assumption.  This means that
\begin{align*}
e^{-T\theta \delta}e^{-\rho(z,T\theta)(\delta-k(z,T\theta))}  \lesssim \muball{z}{T\theta}  &\leq \mu (\cup_{i \in Y_1^0} B(y_i,e^{-T}))\\
&\lesssim \sum_{i \in Y_1^0}  (e^{-T})^{\delta} e^{-\rho(y_i,T)(\delta-k(y_i,T))}\\
&\leq \vert Y_1^0 \vert (e^{-T})^{\delta} e^{\epsilon T(1-\theta)(\kmax-\delta)}
\end{align*}
using $\rho(y_i,T) \leq \epsilon T(1-\theta)$ for $i \in Y_1^0$ and $\delta < \kmax$. Therefore
\begin{align}
\vert Y_1 \vert  & \geq \vert Y_1^0 \vert  \gtrsim \left(e^{T(1-\theta)}\right)^{\delta} e^{\epsilon T(1-\theta)(\delta-\kmax)} e^{-T\theta(\delta-\kmin)}
=\left(e^{T(1-\theta)}\right)^{\delta -\frac{\theta}{1-\theta}(\delta-\kmin)+\epsilon (\delta-\kmax)}.\label{lbound2}
\end{align}
At least one of (\ref{lbound1}) and (\ref{lbound2}) must hold, so we have
\begin{equation*}
\vert Y \vert \geq \vert Y_0 \vert + \vert Y_1 \vert \gtrsim \left(e^{T(1-\theta)}\right)^{\delta -\frac{\theta}{1-\theta}(\delta-\kmin)+\epsilon (\delta-\kmax)}
\end{equation*}
which proves \[\lowspec \lset \geq \delta -\frac{\theta}{1-\theta}(\delta-\kmin)+\epsilon (\delta-\kmax)\]
and letting $\epsilon \rightarrow 0$ proves the desired lower bound.

 \textit{Upper bound}: We show \[\lowspec\lset \leq  \delta - \min\left\{1,\frac{\theta}{1-\theta}\right\}(\delta-\kmin).\]
In order to obtain the upper bound in this case, we require the following technical lemma due to Bowditch \cite{BO}. Switch to the upper half space model $\mathbb{H}^{d+1}$, and let $p \in P$ be a parabolic fixed point with rank $\kmin$. We may assume by conjugation that $p = \infty$. A consequence of geometric finiteness is that the limit set can be decomposed into the disjoint union of a set of conical limit points and a set of bounded parabolic fixed points (this result was proven in dimension 3 partially by Beardon and Maskit \cite{BM} and later refined by Bishop \cite{BJ2}, and then generalised to higher dimensions by Bowditch \cite{BO}). In particular, $p=\infty$ is a bounded parabolic point, and applying \cite[Definition, Page 272]{BO} gives the following lemma.
\begin{lem}\label{Bow}
There exists $\lambda>0$ and a $\kmin$-dimensional linear subspace $V \subseteq \mathbb{R}^d$ such that $\lset \subseteq V_\lambda \cup \{\infty\}$, where $V_\lambda = \{x \in \mathbb{R}^d \mid \inf_{y \in V} ||x-y|| \leq \lambda\}$.
\end{lem}
Note that by Lemma \ref{Bow}, we immediately get $\lowspec \lset \leq \kmin$ for $\theta \geq 1/2$, so we may assume that $\theta \in (0,1/2)$.

We proceed in a similar manner to the lower bound in Section \ref{AsospecLimit2}. Let $p \in \lset$ be a parabolic fixed point with $k(p) = \kmin$, and then switch to the upper half-space model $\mathbb{H}^{d+1}$, and assume that $p = \infty$. Similar to the lower bound in Section \ref{AsospecLimit2}, we may assume that we have a $\mathbb{Z}^{\kmin}$ lattice as a subset of $\lset$, which we denote by $Z$. We partition $Z$ into $\{Z_k\}_{k \in \mathbb{N}}$, where
\[Z_k = \{z \in Z \mid 10(k-1) \leq \size{z} < 10k\} \]
noting that 
\begin{equation}\label{ZK2}
\size{Z_k} \approx k^{\kmin-1}.
\end{equation}
We let $\phi$ denote the Cayley transformation mapping the upper half space model to the Poincar\'e ball model.  Then, as before, there are positive constants $C_1$ and $C_2$ such that  if $z \in Z_k$ for some $k$, then
\begin{align*}
\frac{1}{C_1k} \leq \vert \phi(z) - p_1 \vert \leq \frac{C_1}{k} \qquad 
\text{and} \qquad  \frac{1}{C_2k^2} \leq \vert \phi(B(z,{1}/{3})) \vert \leq \frac{C_2}{k^2},
\end{align*}
where $p_1 = \phi(\infty)$. We let $T>0$, and then by Lemma \ref{Bow}, we may choose a constant $C_3$ such that
\[ B(p_1,e^{-T\theta}) \subseteq \bigcup\limits_{\substack{z \in Z_k  \\ {C_3}/{k^2} >  e^{-T} \\ \vert \phi(z) - p_1 \vert < e^{-T\theta}}} B(\phi(z),C_3/k^2).\]
Then we have
\begin{equation}\label{pack1}
M_{e^{-T}}(B(p_1,e^{-T\theta} \cap \lset) \lesssim \sum_{\substack{z \in Z_k  \\ {C_3}/{k^2} >  e^{-T} \\ \vert \phi(z) - p_1 \vert < e^{-T\theta}}} M_{e^{-T}}(B(\phi(z),C_3/k^2) \cap \lset).
\end{equation}
We wish to estimate $M_{e^{-T}}(B(\phi(z),C_3/k^2) \cap \lset)$ from above.

Let $k \in \mathbb{N}$ satisfy the conditions above, let $\epsilon \in (0,1)$, write $t = \text{log}(k^2/C_3)$. Write $K$ to denote the set of $k$ such that $T-t \leq \max\{\epsilon^{-1},\text{log}10\}$. Then note that
\[\frac{e^{-t}}{e^{-T}} \lesssim 1\]
which gives 
\[M_{e^{-T}}(B(\phi(z),C_3/k^2) \cap \lset) \lesssim 1.\]
Furthermore, using the fact that $C_3/k^2 > e^{-T}$ and the fact that we are summing in $\kmin$ directions, we have
\[\size{K} \lesssim e^{\frac{T\kmin}{2}} \leq e^{T(1-\theta)\kmin} \leq e^{T(1-\theta)\delta}.\]
We now assume that $T-t \geq \max\{\epsilon^{-1},\text{log}10\}$. 

We let $\{x_i\}_{i \in X}$ denote a centred $e^{-T}$-packing of $B(\phi(z),e^{-t}) \cap \lset$ of maximal cardinality, and decompose X as
\begin{equation*}
X = X_0 \cup X_1 \cup \sum\limits_{n=2}^{\infty} X_n
\end{equation*}
where
\begin{align*}
X_0 &= \{i \in X \mid (x_i)_T \in H_p \ \text{with} \ \size{H_p} \geq 10e^{-t}\}\\
X_1 &= \{i \in X \setminus X_0 \mid \rho(x_i,T) \leq \epsilon(T-t) \}\\
X_n &= \{i \in X \setminus (X_0 \cup X_1) \mid n-1 < \rho(x_i,T) \leq n\}.
\end{align*}
For the details on estimating the cardinalities of $X_0$, $X_1$ and $X_n$, we refer the reader to \cite[Section 5.1]{Fr1}.

Suppose that $i \in X_0$. Then $(x_i)_T \in H_p$ for some unique parabolic fixed point $p$ with $\size{H_p} \geq 10e^{-t}$. If $\delta > k(p)$, then it can be shown that
\[\size{X_0} \lesssim \left(\frac{e^{-t}}{e^{-T}} \right)^\delta \]
and if $\delta \leq k(p)$, then it can be shown that
\begin{align*}
\size{X_0} &\lesssim \left(\frac{e^{-t}}{e^{-T}} \right)^\delta \max_{i \in X_0} e^{(\rho(\phi(z),t)-\rho(x_i,T))(k(p)-\delta)}  
 \lesssim \left(\frac{e^{-t}}{e^{-T}} \right)^\delta \max_{i \in X_0} e^{\rho(\phi(z),t)(k(p)-\delta)}.
\end{align*}
In a similar manner to the lower bound in Section \ref{AsospecLimit2}, we can show that $\rho(\phi(z),t) \leq C_4$ for some constant $C_4$, see \eqref{rhoc4}. Therefore, we have
\begin{align*}
\size{X_0} \lesssim \left(\frac{e^{-t}}{e^{-T}} \right)^\delta  e^{C_4(\kmax-\delta)} 
\lesssim \left(\frac{e^{-t}}{e^{-T}} \right)^\delta.
\end{align*}
Similarly, for $i \in X_1$, we have
\begin{align*}
\size{X_1} \lesssim \left(\frac{e^{-t}}{e^{-T}} \right)^{\delta+\epsilon(\delta-\kmin)} e^{\rho(\phi(z),t)(k(p)-\delta)} \lesssim \left(\frac{e^{-t}}{e^{-T}} \right)^{\delta+\epsilon(\delta-\kmin)}
\end{align*}
and if $i \in X_n$ for some $n \geq 2$, then we have
\begin{equation*}
\size{X_n} \lesssim \epsilon^{-1}n e^{-n\delta}   \left(\frac{e^{-t}}{e^{-T}} \right)^{\delta}.
\end{equation*}
It follows that
\begin{align*}
\size{X} = \size{X_0} + \size{X_1} + \sum\limits_{n=2}^{\infty} \size{X_n} 
&\lesssim \left(\frac{e^{-t}}{e^{-T}} \right)^{\delta+\epsilon(\delta-\kmin)} + \sum\limits_{n=2}^{\infty} \epsilon^{-1}n e^{-n\delta}   \left(\frac{e^{-t}}{e^{-T}} \right)^{\delta}\\
&\lesssim \left(\frac{e^{-t}}{e^{-T}} \right)^{\delta+\epsilon(\delta-\kmin)} + \epsilon^{-1} \left(\frac{e^{-t}}{e^{-T}} \right)^{\delta}\\
&= \left(\frac{C_3/k^2}{e^{-T}} \right)^{\delta+\epsilon(\delta-\kmin)} + \epsilon^{-1} \left(\frac{C_3/k^2}{e^{-T}} \right)^{\delta}
\end{align*}
and so by (\ref{pack1}),
\begin{align*}
 M_{e^{-T}}(B(p_1,e^{-T\theta} \cap \lset) 
&\lesssim \sum_{\substack{z \in Z_k  \\ {C_3}/{k^2} >  e^{-T} \\ \vert \phi(z) - p_1 \vert < e^{-T\theta}}}\left(\frac{C_3/k^2}{e^{-T}} \right)^{\delta+\epsilon(\delta-\kmin)} + \epsilon^{-1} \left(\frac{C_3/k^2}{e^{-T}} \right)^{\delta}\\
&\lesssim (1+\epsilon^{-1}) e^{T\delta+\epsilon(\delta-\kmin)} \sum_{k = \lceil e^{T\theta}/C_1 \rceil}^{\lfloor \sqrt{C_3}e^{T/2} \rfloor} \vert Z_k \vert  \left(C_3/{k^2}\right)^{\delta+\epsilon(\delta-\kmin)} \\
&\lesssim  (1+\epsilon^{-1}) e^{T\delta+\epsilon(\delta-\kmin)}  \sum_{k = \lceil e^{T\theta}/C_1 \rceil}^{\lfloor \sqrt{C_3}e^{T/2} \rfloor} k^{\kmin-1-2\delta-2\epsilon(\delta-\kmin)} \ \text{by} \ (\ref{ZK2}) \\
&\lesssim  (1+\epsilon^{-1}) e^{T\delta+\epsilon(\delta-\kmin)} \int_{e^{T\theta}/C_1 }^{ \sqrt{C_3}e^{T/2} } x^{\kmin-1-2\delta-2\epsilon(\delta-\kmin)} dx\\
&\lesssim  (1+\epsilon^{-1}) e^{T\delta+\epsilon(\delta-\kmin)}  e^{T\theta(\kmin-2\delta-2\epsilon(\delta-\kmin))} dx\\
&\lesssim  (1+\epsilon^{-1}) \left(e^{T(1-\theta)} \right)^{\delta-\frac{\theta}{1-\theta}(\delta-\kmin) + \epsilon(\delta-\kmin)} e^{-T\theta \epsilon(\delta-\kmin)}\\
&\lesssim  (1+\epsilon^{-1}) \left(e^{T(1-\theta)} \right)^{\delta-\frac{\theta}{1-\theta}(\delta-\kmin) + \epsilon(\delta-\kmin)}
\end{align*}
which proves \[\lowspec \lset \leq \delta-\frac{\theta}{1-\theta}(\delta-\kmin) + \epsilon(\delta-\kmin)\]
and letting $\epsilon \rightarrow 0$ proves the upper bound.

\subsubsection{When $\delta \geq \kmax$}\label{LowspecLimit3} 
 We show \[\lowspec\lset =  \delta - \min\left\{1,\frac{\theta}{1-\theta}\right\}(\delta-\kmin).\]
The lower bound follows from the fact that 
\[\lowspec \lset \geq \lowspec \mu = \delta - \min\left\{1,\frac{\theta}{1-\theta}\right\}(\delta-\kmin)\]
and therefore it suffices to prove the upper bound. The case when $\theta \geq {1}/{2}$ is an immediate consequence of Lemma \ref{Bow}, and so we may assume that $\theta \in (0,{1}/{2})$. Let $p \in P$ be a parabolic point with rank $\kmin$, and let $T>0$ be sufficiently large. We can then attain an upper bound for 
$N_{e^{-T}}(B(p,e^{-T\theta}) \cap \lset)$ by first covering $B(p,e^{-T\theta}) \cap \lset$ with a $\kmin$-dimensional collection of balls of radius $e^{-2T\theta}$ using Lemma \ref{Bow}, then covering each of those balls by balls of radius $e^{-T}$. Using the fact that $\aso \lset = \delta$, we have
\begin{align*}
N_{e^{-T}}(B(p,e^{-T\theta}) \cap \lset) \lesssim N_{e^{-2T\theta}}(B(p,e^{-T\theta}) \cap \lset) \left(\frac{e^{-2T\theta}}{e^{-T}} \right)^{\delta}
&\lesssim   \left(\frac{e^{-T\theta}}{e^{-2T\theta}} \right)^{\kmin} \left(\frac{e^{-2T\theta}}{e^{-T}} \right)^{\delta}\\
&= e^{T(1-\theta)\delta + T \theta (\kmin-\delta)} \\
&= \left(e^{T(1-\theta)}\right)^{\delta - \frac{\theta}{1-\theta}(\delta-\kmin)}
\end{align*}
which proves 
\[\lowspec \lset \leq \delta - \frac{\theta}{1-\theta}(\delta-\kmin) \]
as required.

\begin{center} \textbf{Acknowledgements}
\end{center}
JMF was financially supported by an \textit{EPSRC Standard Grant} (EP/R015104/1) and a \textit{Leverhulme Trust Research Project Grant} (RPG-2019-034). LS was financially supported by the University of St Andrews.

\bibliographystyle{apalike}
\addcontentsline{toc}{section}{References}
\bibliography{Kleinian}
\end{document}